\documentclass{amsart}
\usepackage[utf8]{inputenc}
\usepackage{needspace}
\usepackage{amssymb,bm,amsfonts,amsmath,amsthm,enumerate,mathtools,comment}
\usepackage[all]{xy}
\makeatletter
\@namedef{subjclassname@2020}{\textup{2020} Mathematics Subject Classification}
\makeatother
\usepackage[colorlinks]{hyperref}
\hypersetup{
 colorlinks=true,
 linkcolor=blue,
 filecolor=magenta,
 urlcolor=cyan,
}
\usepackage{wrapfig,tikz, tikz-cd}
\usetikzlibrary{arrows}
\usepackage[capitalise]{cleveref}
\crefformat{equation}{(#2#1#3)}
\crefrangeformat{equation}{(#3#1#4--#5#2#6)}
\crefformat{enumi}{(#2#1#3)}
\crefrangeformat{enumi}{(#3#1#4--#5#2#6)}
\newtheorem{theorem}{Theorem}[section]
\newtheorem{lemma}[theorem]{Lemma}
\newtheorem{proposition}[theorem]{Proposition}
\newtheorem{corollary}[theorem]{Corollary}
\newcounter{intro}

\newtheorem{introthm}[intro]{Theorem}
\newtheorem{introcor}[intro]{Corollary}
\newtheorem{introquestion}[intro]{Question}

\theoremstyle{definition}
\newtheorem{definition}[theorem]{Definition}
\newtheorem{example}[theorem]{Example}
\newtheorem{remark}[theorem]{Remark}
\newtheorem{observation}[theorem]{Observation}

\newtheorem{chunk}[theorem]{}

\newtheorem{question}[theorem]{Question}

\newtheorem{thm}{Theorem}[subsection]

\theoremstyle{definition}

\newtheorem{setup}{Setup}

\newcounter{constructionstep}
\newcommand{\constructionstep}[1]{%
  \par\addvspace{\medskipamount}\noindent
  \stepcounter{constructionstep}\textit{\Roman{constructionstep}. #1.}\enspace\ignorespaces}

\newcommand{\pdim}{{\operatorname{pdim}}}

\newcommand{\Spec}{{\operatorname{Spec}}}

\newcommand{\Hom}{{\operatorname{Hom}}}

\newcommand{\Ext}{{\operatorname{Ext}}}
\newcommand{\Tor}{{\operatorname{Tor}}}
\DeclareMathOperator{\Soc}{Soc}
\newcommand{\RHom}{\operatorname{\mathsf{RHom}}}

\renewcommand{\H}{\operatorname{H}}

\DeclareMathOperator{\coker}{coker}
\DeclareMathOperator{\depth}{depth}

\newcommand{\V}{{\rm{V}}}

\renewcommand{\ll}{\ell\ell}

\newcommand{\ann}{{\operatorname{ann}}}

\newcommand{\shift}{{\scriptstyle\mathsf{\Sigma}}}
\newcommand{\lotimes}{\otimes^{\sf L}}

\DeclareMathOperator{\rank}{rank}

\newcommand{\m}{\mathfrak{m}}
\newcommand{\p}{\mathfrak{p}}

\DeclareMathOperator{\onto}{\twoheadrightarrow}

\DeclareMathOperator{\Supp}{Supp}

\DeclareMathOperator{\gr}{gr}
\DeclareMathOperator{\Kos}{Kos}

\newcommand{\n}{\mathfrak n}

\makeatletter

\def\l@subsection{\@tocline{2}{0pt}{2pc}{5pc}{}}
\makeatother

\title{Modules with unbalanced support}
\date{}

\author[N.~KC]{Nawaj KC}
\address{Department of Mathematics,
University of Utah, Salt Lake City, UT 84112, U.S.A.}
\email{nawaj.kc@utah.edu}

\subjclass[2020]{13D07, 13D22 (Primary), 13D05, 13H10}

\begin{document}
\begin{abstract}
A finitely generated module $M$ over a Noetherian local ring $(R,\m)$
has \emph{unbalanced support} if its support has positive dimension
and meets a top-dimensional irreducible component only at the closed
point. From one such module we construct a counterexample to the
Peskine--Szpiro dimension inequality in which both modules have finite
projective dimension. This also disproves strong positivity for
intersection multiplicities. A related patching construction yields a
family of counterexamples to strong vanishing with both modules
Cohen--Macaulay.

Suppose $R$ is Gorenstein and $M$ has finite projective dimension and
unbalanced support. Let $T=\H^0_{\m}(M)$ denote its finite-length
torsion. Write the spectrum as a union of two closed subsets, called
wings. A length formula relating the spine $M/T$ to the collapsed wing
rules out unbalanced support across linked Gorenstein wings. We
construct examples with just one Gorenstein wing. If the collapsed
wing is Cohen--Macaulay, its index bounds the Loewy length of $T$
from below.
\end{abstract}

\maketitle
\begingroup
\fontsize{9}{10}\selectfont
\vspace*{-15pt}
\tableofcontents
\removelastskip 
\endgroup
\newpage
\section{Introduction}

Serre \cite{Ser} proved that if the supports of two finitely generated
modules $M$ and $N$ over a regular local ring $R$ meet only at the
closed point, then
\[
    \dim M+\dim N\leq\dim R.
\]
Peskine and Szpiro \cite{PS1} predicted that the same inequality would
hold over an arbitrary Noetherian local ring when $M$ has finite
projective dimension. Consider what this predicts when $N=R/\p$ is
the coordinate ring of a top-dimensional irreducible component. If
$\Supp_R M$ meets that component only at the closed point, then $M$
would have finite length. The examples in \cite{Ma} show that $M$ can
have finite projective dimension and still have positive-dimensional
support on another component.

We say that a module has \emph{unbalanced support} if its support has
positive dimension and meets a top-dimensional irreducible component
only at the closed point. We write $\Spec \, R$ as the union of two
closed subsets, which we call wings. On the \emph{supporting wing}
$A=R/I$, the restriction of $M$ has positive dimension. On the
\emph{collapsed wing} $B=R/J$, which has the same dimension as $R$,
the restriction has finite length. When $M$ has finite projective
dimension, the New Intersection Theorem \cite{PS1,H,Ro1} and the
Auslander--Buchsbaum formula give
\begin{equation}\label{eq:nit-reduction}
    R\text{ is Cohen--Macaulay},\qquad T:=\H^0_{\m}(M)\neq0.
\end{equation}
How do the restrictions to the two wings appear in this torsion
submodule, which is supported at a single point?

Our starting point in Section~\ref{sec:geometry} is the quotient
$M/T$, which we call the \emph{spine}. It has the same support as $M$
and has positive depth. We show that at positive depth the first Tor
with the supporting wing is nonzero, and give conditions for the
second Tor with the collapsed wing to be nonzero as well. When
$\Tor_1^R(M,A)=0$, the torsion of $M$ surjects onto the torsion of
its restriction to $A$. The kernel is $I\otimes_B(M\otimes_R B)$,
the restriction to the collapsed wing tensored with the ideal of the
common subscheme in $B$. We first see this in a variation of the
example in \cite{Ma}.

\subsection{An isolated intersection of unexpected dimension}

Glue a regular threefold $A_0=C_0[[u]]$ and a three-dimensional
quadric $B_0$ along a regular surface $C_0$. In our variation the
regular threefold is the supporting wing and the quadric is the
collapsed wing. On $A_0$ we take
$k[[u]]\oplus k^{\oplus5}$, the coordinate ring of the $u$-axis
together with five copies of the residue field at the origin. On the
quadric we take a finite-length module $M_{B_0}$ from the construction
of Dutta, Hochster, and McLaughlin \cite{DHM}. The two modules have
the same derived restriction to $C_0$. Using the patching method of
\cite[Section~2.3]{Ma}, we obtain a module $M_0$ of finite projective
dimension over $R_0=A_0\times_{C_0}B_0$.

The spine of $M_0$ is $k[[u]]$, so its support is the $u$-axis.
At the origin, both wings contribute to its torsion $T_0$.
Writing $P_0$ for the ideal of the common surface in $B_0$,
Theorem~\ref{prop:structure} gives
\[
    0\longrightarrow P_0\otimes_{B_0}M_{B_0}
    \longrightarrow T_0\longrightarrow k^{\oplus5}\longrightarrow0.
\]
The quotient is precisely the torsion we placed on the regular
threefold. The kernel comes from the module on the quadric, tensored
with the ideal of the surface along which we glued. We carry out this
calculation in Section~\ref{subsec:variation}.

The axis meets the quadric only at the closed point, giving
\[\dim M_0+\dim B_0=1+3>3=\dim R_0.\] Here $B_0$ has infinite
projective dimension over $R_0$. Does the dimension inequality hold
if both modules have finite projective dimension? This is Question~3
in \cite{Ma}. To give a counterexample, we use $M_0$ twice.

\begin{introthm}\label[theorem]{thm:intro-fivefold-pair}
There is a five-dimensional complete reduced Cohen--Macaulay local
ring $(R,\m)$ admitting two three-dimensional finitely generated
modules $U_1,U_2$ of projective dimension three such that
$\Supp_R U_1\cap\Supp_R U_2=\{\m\}$ and
\[
    \dim(U_1)+\dim(U_2)=3+3>5=\dim(R).
\]
\end{introthm}

The ring $R$ is the fiber product of two five-dimensional quadric
hypersurfaces over a four-dimensional regular local ring. It contains
two three-dimensional quadrics, with coordinate rings $H_1$ and $H_2$,
that meet only at the closed point. Under a base change $R_0\to R$,
the coordinate ring of the $u$-axis becomes $H_1$, and $M_0$ becomes
a module $U_1$ of finite projective dimension supported on this quadric.
The wing dimensions $(1,0)$ of $M_0$ become $(3,2)$ for $U_1$.
The involution interchanging the wings gives a second base change and
a module $U_2$ supported on the other quadric, with wing dimensions
$(2,3)$. This proves Theorem~\ref{thm:intro-fivefold-pair};
see Theorem~\ref{thm:fivefold-pair}.

We can also follow the two pieces of $T_0$ through these base changes.
Its image in each $U_i$ is a two-dimensional correction term $L_i$
fitting into an exact sequence
\[
    0\longrightarrow L_i\longrightarrow U_i
    \longrightarrow H_i\longrightarrow0.
\]
The contribution from $P_0\otimes_{B_0}M_{B_0}$ is now supported
on a plane in the opposite wing. The quotient of $L_i$ by this
submodule consists of five copies of the coordinate ring of a
hyperplane section of $H_i$. Thus the torsion sequence on the
threefold also describes the modules on the fivefold;
see Proposition~\ref{prop:fivefold-correction}.
\subsection{Intersection multiplicity conjectures}

Let $M$ and $N$ be nonzero finitely generated modules over a local ring
$(R,\m)$ whose supports meet only at the closed point. If at least one
of the modules has finite projective dimension, their intersection
multiplicity is
\[
    \chi^R(M,N)=\sum_{i\geq0}(-1)^i\ell_R\bigl(\Tor_i^R(M,N)\bigr).
\]
Generalized positivity asks whether $\chi^R(M,N)>0$ when
$\dim M+\dim N=\dim R$. Generalized vanishing asks whether
$\chi^R(M,N)=0$ when $\dim M+\dim N<\dim R$. We consider these
questions under the stronger hypothesis that both modules have finite
projective dimension. This was a question of Roberts \cite[Question ~6]{Ro3}.

We answer both questions in the negative. In fact, the negative answer
to positivity follows immediately from Theorem~\ref{thm:intro-fivefold-pair}. 
The proof is left to the reader. 

\begin{introcor}\label[corollary]{cor:intro-positivity}
Over the ring and modules of Theorem~\ref{thm:intro-fivefold-pair}, let $M = U_1/(a)U_1$
for a nonzerodivisor $a\in\m$ on $U_1$, and let $N = U_2$. 
Then $M$ and $N$ are modules of finite projective dimension such that $\Supp M \cap \Supp N = \{\m\}$,

\[ \dim M + \dim N  = 2 + 3 = \dim R, \qquad \chi^R(M, N) = 0.\]

\end{introcor}

A variation of the construction in Theorem~\ref{thm:intro-fivefold-pair} gives a counterexample to vanishing. In fact, we give a whole family of counterexamples. It is worth remarking that both modules are themselves
Cohen--Macaulay.

\begin{introthm}\label[theorem]{thm:intro-vanishing}
Let $H$ be a $d$-dimensional complete Cohen--Macaulay local ring with
coefficient field $k$, and let $C_0$ be a Cohen--Macaulay quotient of
dimension $d-1$. Suppose that $L$ is a finite-length $H$-module of
finite projective dimension with \[ r=\chi^H(L,C_0)\ne 0.\]
Then there is a complete Cohen--Macaulay local ring $(S,\mathfrak n)$
of dimension $2d-1$ and two Cohen--Macaulay $S$-modules $M_1,M_2$,
each of dimension $d-1$ and projective dimension $d$, such that $\Supp_S M_1\cap\Supp_S M_2=\{\mathfrak n\}$, and
\[
 \dim M_1 + \dim M_2 = 2d-2 < 2d-1 = \dim S,
  \qquad
  \chi^S(M_1,M_2)=-r^2.
\]
\end{introthm}

For the example in \cite{DHM} considered above, we have $r =1$, and thus $\dim M_1 + \dim M_2 = 2+2 < 5 =\dim S$ 
and $\chi^S(M_1, M_2) = -1$.
The constructions in \cite{DHM,RS} give many modules satisfying the nonvanishing hypothesis above, 
and Miller and Singh \cite{MS} give an analogous example in codimension two. Thus the vanishing conjecture 
fails even when both modules have finite
projective dimension, answering a question of Roberts
\cite[Question~6(i)]{Ro3} in the negative.
Since vanishing holds in this setting whenever one module has finite
length, Roberts wrote that ``one must look elsewhere'' \cite[p.~214]{Ro3}.
All along, the Dutta--Hochster--McLaughlin example was just one
fiber-product construction away from such a counterexample.


\subsection{A Gorenstein construction}

We will now return to the study of modules of finite projective
dimension and unbalanced support over Gorenstein ambient rings.
The two wings of $R_0$ are both Gorenstein. Can this happen when
the ambient ring is Gorenstein as well?

Suppose that $R$ is Gorenstein of dimension $d$, that its two
Cohen--Macaulay wings are linked, and that the supporting wing $A$
is Gorenstein. Write $N=M/T$ for the spine and $M_B=M\otimes_R B$
for the restriction to the collapsed wing. The spine meets the common
subscheme $\Spec \, C$, where $C=A\otimes_R B$, in a module of
positive length. The torsion--spine balance computes this length
on the opposite wing:
\[
    \ell(N\otimes_A C)
    =\ell(M_B)-\ell\bigl(\Ext_B^d(M_B,B)\bigr)>0.
\]
If $B$ were also Gorenstein, $M_B$ and its top Ext dual would have
the same length.

Roberts' construction \cite{Ro2} supplies a Cohen--Macaulay wing
with a module of length $31$ whose top Ext dual has length $29$.
Taking this module on the collapsed wing, we choose a spine on the
supporting wing with
\[
    \ell(N\otimes_A C)=31-29=2.
\]
We then add the finite-length torsion needed to patch the two
restrictions. A similar balance with the roles of the wings reversed
leads to a second module over the same ring.

\begin{introthm}\label[theorem]{thm:intro-gorenstein-example}
\begin{enumerate}[(1)]
\item If $R$ is a Gorenstein local ring with two linked Gorenstein
wings, no finitely generated $R$-module of finite projective dimension
has unbalanced support across them.
\item For every field $k$ of characteristic different from two, there
is a three-dimensional complete reduced Gorenstein local ring $R$
with residue field $k$ and exactly two irreducible components. Their
coordinate rings are linked wings, one a hypersurface and the other
Cohen--Macaulay of type three. For either prescribed wing, there is
a one-dimensional $R$-module of finite projective dimension whose
support collapses on that wing.
\end{enumerate}
\end{introthm}

\subsection{The size of the finite-length torsion}

The contribution from the collapsed wing also constrains the size
of the torsion. Over a $d$-dimensional Gorenstein ring, let $B=R/J$
be a full-dimensional Cohen--Macaulay quotient on which the support
collapses. Proposition~\ref{prop:intersection-dual} identifies
$(0:_T J)$ with the Matlis dual of the top Ext dual
$D=\Ext_B^d(M_B,B)$. For linked Cohen--Macaulay wings,
Corollary~\ref{cor:torsion-socle} identifies this submodule with $IM$.
The module $D$ has finite length and finite projective dimension
over $B$, and Matlis duality preserves its Loewy length.

Recall that the Loewy length $\ll_R(T)$ is the smallest $n$ for
which $\m^nT=0$. The index of $B$ measures its singularity
cohomologically and equals one precisely when $B$ is regular
\cite{Mar}.

\begin{introthm}\label[theorem]{thm:intro-torsion-bound}
Let $(R,\m)$ be a Gorenstein local ring and let $B=R/J$ be
Cohen--Macaulay with $\dim B=\dim R$. Let $M\neq0$ be a finitely
generated $R$-module of finite projective dimension whose support
meets $\Spec \, B$ only at the closed point. Then
\[
    \ll_R\bigl(\H^0_{\m}(M)\bigr)\geq\operatorname{index}(B).
\]
If $\gr_{\m_B}(B)$ is Cohen--Macaulay, then also
\[
    \ll_R\bigl(\H^0_{\m}(M)\bigr)
    \geq\operatorname{reg}\bigl(\gr_{\m_B}(B)\bigr)+1,
\]
where $\operatorname{reg}(-)$ denotes Castelnuovo--Mumford regularity.
\end{introthm}

When $B$ is a complete intersection, local duality identifies
$D^\vee$ with $M_B$. Its quotient in $T$ has positive length
whenever $J\neq0$. Taking lengths and applying the Dutta multiplicity
bound from the recent proof of Lech's conjecture \cite{OAI} gives
the following corollary, with the mixed characteristic bound supplied
by \cite[Theorem~1.2 and Remark~2.10]{MaLech}. Here $e(B)$ denotes
the Hilbert--Samuel multiplicity at the maximal ideal.

\begin{introcor}\label[corollary]{cor:intro-ci-length}
Under the hypotheses of Theorem~\ref{thm:intro-torsion-bound},
suppose $B$ is a complete intersection and put $T=\H^0_{\m}(M)$.
Then $\ell_R(T)\geq\ell_B(M_B)\geq e(B)$. If $J\neq0$, then
\[
    \ell_R(T)\geq\ell_B(M_B)+1\geq e(B)+1.
\]
\end{introcor}

Depth zero alone allows $T=k$, even for a positive-dimensional module
of finite projective dimension over a singular Gorenstein ring, as
Example~\ref{ex:full-support-torsion} shows. With unbalanced support,
the contribution from the collapsed wing contains the entire socle of
$T$ in the linked Gorenstein setting. The torsion sequence is then an
essential nonsplit extension; see Proposition~\ref{prop:intersection-dual}.
For fixed wing restrictions, we ask in Question~\ref{q:torsion-size}
which extension classes arise from gluing. With only the ambient
ring fixed, we ask the following.

\begin{introquestion}\label[question]{q:intro-torsion}
Which finite-length $R$-modules occur as $\H^0_{\m}(M)$ when $M$ has
finite projective dimension and unbalanced support?
\end{introquestion}

\section{Homological signals and torsion}
\label{sec:geometry}

Let $(R,\m)$ be a Noetherian local ring and let $M$ be a finitely
generated $R$-module. What does unbalanced support force in a free
resolution of $M$? We study this question through the two-wing
framework of Setup~\ref{setup:two-wings} below. Every module with
unbalanced support admits such a setup. We choose the wings from a
primary decomposition of $(0)$ in $R$.

In positive depth, the first Tor with the supporting wing is nonzero,
and we give criteria for the second Tor with the collapsed wing to
be nonzero as well. In depth zero, however, finite-length torsion
can allow both groups to vanish. We study what the torsion must look
like for this to happen.

\begin{definition}
We say that $M$ has \emph{unbalanced support} if $\dim_R M\geq1$ and
\[
 \Supp_R M\cap V(\p)=\{\m\}
\]
for some prime $\p$ with $\dim R/\p=\dim R$.
\end{definition}

\begin{samepage}
\begin{setup}\label{setup:two-wings}
Fix nonzero ideals $I,J\subseteq R$ with $I\cap J=0$, and set
\[
 A=R/I,\qquad B=R/J,\qquad C=R/(I+J).
\]
Write $M_A=M\otimes_R A$ and $M_B=M\otimes_R B$. Assume that
$\dim B=\dim R$ and
\begin{equation}\label{eq:unbalanced}
 \dim M_A\geq1,\qquad \dim M_B=0.
\end{equation}
We call $A$ the \emph{supporting wing} and $B$ the
\emph{collapsed wing}. Since $I\cap J=0$ and $C=A\otimes_R B$, the
diagonal map and the difference of the quotient maps give a short exact
sequence
\begin{equation}\label{eq:ring-patching}
 0\longrightarrow R\longrightarrow A\oplus B\longrightarrow C
 \longrightarrow0.
\end{equation}
Thus $R$ is the fiber product of $A$ and $B$ over $C$.
\end{setup}
\end{samepage}

To choose the wings, fix $\p$ as in the definition and take an
irredundant primary decomposition
\[
 (0)=Q_1\cap\cdots\cap Q_s,\qquad \sqrt{Q_1}=\p.
\]
Taking $J=Q_1$ and $I=\bigcap_{\nu=2}^s Q_\nu$ gives the required
ideals. The choice of wings need not be unique: other components
or groupings of primary components may satisfy the same conditions.
We work with a fixed choice.

The basic geometric observation is
\begin{equation}\label{eq:geometric-torsion}
 IM\subseteq \H^0_{\m}(M).
\end{equation}
Indeed, $JI=0$, so
$\Supp_R(IM)\subseteq\Supp_R M\cap V(J)=\{\m\}$.

Throughout, complexes are homologically graded, with
$\H_i(\shift X)=\H_{i-1}(X)$.

\subsection{Positive-depth signals}
\label{subsec:signals}
\label{subsec:mixed-syzygies}

We first ask which positive Tor groups unbalanced support forces to be
nonzero when $M$ has positive depth.

\begin{proposition}\label[proposition]{lem:positive-depth}
If $M$ is a finitely generated $R$-module of positive depth with unbalanced support and supporting wing $A=R/I$, then
\[
 IM=0,\qquad \Tor_1^R(M,A)\cong I\otimes_R M\neq0.
\]
\end{proposition}

\begin{proof}
The first assertion follows from \cref{eq:geometric-torsion}.
Tensoring $0\to I\to R\to A\to0$ with $M$ gives the isomorphism.
Modulo $\m$, the right hand side is the tensor product
of $k$-vector spaces $I/\m I$ and $M/\m M$, both nonzero due to Nakayama's lemma.
\end{proof}

\begin{question}\label[question]{q:second-tor}
Let $M$ be as in \cref{lem:positive-depth}, with collapsed wing
$B=R/J$. Must $\Tor_i^R(M,B)$ be nonzero for some $i>0$?
\end{question}

Principal quotients give one affirmative case. For $g\in R$, place
the Koszul complex $\Kos_R(g)$ in degrees one and zero. Its truncation
triangle is
\[
 \shift(0:_R g)\longrightarrow\Kos_R(g)
 \longrightarrow R/(g)\longrightarrow.
\]
Derived tensoring with $M$ gives
\begin{equation}\label{eq:principal-tor}
\begin{aligned}
0\longrightarrow\Tor_2^R(M,R/(g))
&\longrightarrow(0:_R g)\otimes_R M\\
&\xrightarrow{\;\mathrm{mult}\;}(0:_M g)
\longrightarrow\Tor_1^R(M,R/(g))\longrightarrow0.
\end{aligned}
\end{equation}

\begin{theorem}\label[theorem]{thm:principal-signals}
If $M$ is a module of positive depth and $\Supp M \cap \V(J) = \{ \m \}$, then $J$ contains
$M$-regular elements. For every such $g$,
\[
 \Tor_1^R(M,R/(g))=0,\qquad
 \Tor_2^R(M,R/(g))\cong(0:_R g)\otimes_R M\neq0.
\]
In particular, $\Tor_2^R(M,B)\neq0$ when $J$ is principal.
\end{theorem}

\begin{proof}
Since the support of $M$ collapses on $\V(J)$, any associated prime of $M$ containing $J$ would equal $\m$. 
Since $M$ has positive depth, prime avoidance therefore gives an $M$-regular
element in $J$. Apply \cref{eq:principal-tor}; the tensor product is
nonzero by Nakayama's lemma, since $(0:_R g)\supseteq I\neq0$.\qedhere
\end{proof}

A second criterion comes from a presentation $R=S/(I\cap J)$, with
$A=S/I$ and $B=S/J$. By the Cohen structure theorem, such a
presentation with $S$ regular always exists after completion.
Completion preserves the hypotheses and detects Tor nonvanishing
by faithful flatness. For the criterion itself, $(S,\n)$ may be any
Noetherian local ring giving such a presentation. Here $I,J$ are the
inverse images of the wing ideals, and $\bar I,\bar J$ are their
images in $R$.

\begin{theorem}\label[theorem]{thm:mixed-syzygy-nonvanishing}
Retain the positive-depth and support hypotheses of
\cref{lem:positive-depth}, and the presentation above. If
$IJ\nsubseteq\n(I\cap J)$, then $\Tor_2^R(M,B)\neq0$.
\end{theorem}

\begin{proof}
Choose generators $g_1,\ldots,g_s$ of $J$, and let
$\Omega=\ker(R^s\to\bar J)$. There is a well-defined map
\[
 \psi:\Omega\longrightarrow (I\cap J)/\n(I\cap J),
 \qquad (\bar a_j)\longmapsto\Bigl[\sum_j a_jg_j\Bigr].
\]
The sum belongs to $I\cap J$, and changing the lifts changes it by
an element of $J(I\cap J)\subseteq\n(I\cap J)$.
Moreover, $\psi(\m\Omega)=0$.

After reordering the generators, there is $f\in I$ with
$fg_1\notin\n(I\cap J)$. The relation
$\sigma=\bar f e_1\in\Omega$ has nonzero image under $\psi$, so
$\sigma\notin\m\Omega$. By \cref{lem:positive-depth}, $\bar I M=0$.
Dimension shifting identifies $\Tor_2^R(M,B)$ with the kernel of
$M\otimes_R\Omega\to M\otimes_R R^s$.
For $m_0\in M\setminus\m M$, the element $m_0\otimes\sigma$ lies
in this kernel and is nonzero modulo $\m$, since both tensor factors
have nonzero images. It gives the required Tor class.
\end{proof}

A product in $IJ$ that survives modulo $\n(I\cap J)$ thus gives a minimal
relation whose coefficient annihilates $M$. Question~\ref{q:second-tor}
remains open when $\bar J$ is nonprincipal and
$IJ\subseteq\n(I\cap J)$. We now return to $I,J\subseteq R$.

\subsection{The torsion and the spine}
\label{sec:spine}
\label{sec:linked}

For a module $M$ with unbalanced support and depth zero, let $T=\H^0_{\m}(M)$ and
$N=M/T$, so
\begin{equation}\label{eq:torsion-spine}
 0\longrightarrow T\longrightarrow M\longrightarrow N\longrightarrow0.
\end{equation}
We call $T$ the \emph{torsion} and $N$ the \emph{spine}.
The spine has positive depth and the same support as $M$.
By \cref{lem:positive-depth}, $IN=0$ and
$\Tor_1^R(N,A)\cong I\otimes_R N\neq0$. What must the torsion $T$
contain if $\Tor_1^R(M,A)=0$?

\begin{theorem}\label[theorem]{prop:structure}
The modules $M$ and $M_A$ have the same spine:
$N\cong M_A/\H^0_{\m}(M_A)$.
If $\Tor_1^R(M,A)=0$, multiplication gives an exact sequence
\begin{equation}\label{eq:torsion-components}
 0\longrightarrow I\otimes_B M_B\longrightarrow T
 \longrightarrow\H^0_{\m}(M_A)\longrightarrow0.
\end{equation}
\end{theorem}

\begin{proof}
Since $IM\subseteq T$ has finite length, applying local cohomology to
\[ 0\to IM\to M\to M_A\to0 \] gives
$0\to IM\to T\to\H^0_{\m}(M_A)\to0$ and the spine identification.
The vanishing of $\Tor_1^R(M,A)$ identifies $I\otimes_R M$ with
$IM$. Since $JI=0$, one has \[ I\otimes_R M\cong I\otimes_B M_B. \qedhere \]
\end{proof}

Impose now the annihilator conditions
\begin{equation}\label{eq:linkage}
 \ann_R I=J,\qquad \ann_R J=I.
\end{equation}
We call $A$ and $B$ \emph{linked wings}. Either wing may contain
several irreducible components. These conditions hold for partitions of
the minimal primes of a reduced ring and require no Cohen--Macaulay
hypothesis.

\begin{proposition}\label[proposition]{prop:principal}
Assume \cref{eq:linkage}. If $I=(f)$ and $J=(g)$, then every $M$
satisfying \cref{eq:unbalanced} has $\Tor_1^R(M,A)\neq0$.
\end{proposition}

\begin{proof}
The annihilator equalities give alternating resolutions with
differentials $f$ and $g$. For every finite-length $R$-module $E$,
\begin{align*}
 \ell\bigl(\Tor_1^R(E,A)\bigr)
 &=\ell(0:_E f)-\ell(gE)\\
 &=\ell(E)-\ell(fE)-\ell(gE)
 =\ell\bigl(\Tor_1^R(E,B)\bigr).
\end{align*}
Thus one Tor group vanishes if and only if the other does.

By \cref{lem:positive-depth}, $fN=0$ and $\Tor_1^R(N,A)\neq0$.
Since $N/gN$ has finite length and $N$ has positive depth, $g$ is
$N$-regular. Hence $\Tor_2^R(N,A)=(0:_N g)/fN=0$.
Suppose that $\Tor_1^R(M,A)=0$. The long exact sequence from
\cref{eq:torsion-spine} gives $\Tor_1^R(T,A)=0$ and an injection
$\delta_1:\Tor_1^R(N,A)\to T/fT$.
The length equality then gives $\Tor_1^R(T,B)=0$, or $(0:_T g)=fT$.
Under $\Tor_1^R(N,A)\cong N/gN$, the connecting map sends
$[n]$ to $[f\widetilde n]$, where $\widetilde n\in M$ lifts $n$.
But $fM\subseteq T$ and $gfM=0$, so
\[
 \operatorname{im}\delta_1=fM/fT\subseteq(0:_T g)/fT=0.
\]
This contradicts injectivity and $\Tor_1^R(N,A)\neq0$.
\end{proof}

\begin{remark}\label[remark]{rem:gemini}
There is also an Euler-characteristic proof. Since $M/gM$ has finite
length and $\dim M\geq1$, one has $\dim M=1$ and
$\chi(g;M)=\ell(M/gM)-\ell(0:_M g)>0$ by positivity of multiplicity
\cite{Ser}. If $\Tor_1^R(M,A)=0$, then $(0:_M f)=gM$, so
$M/gM\cong fM\subseteq(0:_M g)$, a contradiction.
\end{remark}


\section{Cohen--Macaulay wings and finite projective dimension}
\label{sec:setup}

We now retain the linked two-wing setup of \cref{sec:linked} and assume
that $R$, $A$, and $B$ are Cohen--Macaulay of dimension $d$.
Cohen--Macaulayness is an additional hypothesis on the wings, even when
$R$ is reduced. When $R$ is Gorenstein, linkage makes the
Cohen--Macaulayness of the two wings equivalent \cite{PS2}.
Modules satisfy \cref{eq:unbalanced}; finite projective dimension is
imposed separately.

The partition of the primary components ensures that $I+J$ lies in no
minimal prime of $R$, so $\dim C\leq d-1$. The depth lemma applied to
\cref{eq:ring-patching} gives $\depth C\geq d-1$. Thus the
intersection ring $C$ is Cohen--Macaulay of dimension $d-1$.

A finite free resolution of $M$ restricts on $B$ to a resolution of
a finite-length module. This determines its length on both wings and
forces the Tor vanishing required in \cref{prop:structure}.

\begin{proposition}\label[proposition]{obs:restriction}
Let $R$, $A=R/I$, and $B=R/J$ be Cohen--Macaulay of dimension $d$
in the linked two-wing setup of \cref{sec:linked}. If $M$ has finite
projective dimension over $R$ and satisfies
$\dim M_A\geq1$ and $\dim M_B=0$, then
$M_A$ and $M_B$ have finite projective dimension over $A$ and $B$,
respectively. Moreover,
\[
 \pdim_R M=\pdim_A M_A=\pdim_B M_B=d,
\]
$\depth_R M=\depth_A M_A=0$, and
\[
 \begin{gathered}
 \Tor_i^R(M,A)=\Tor_i^R(M,B)=0\quad\text{for }i\geq1,\\
 \Tor_i^R(M,C)=0\quad\text{for }i\geq2,\qquad
 \Tor_1^R(M,C)\neq0.
 \end{gathered}
\]
\end{proposition}

\begin{proof}
Let $F$ be a finite minimal free resolution of $M$. The determinantal
ideals of its differentials satisfy the grade conditions in the
Buchsbaum--Eisenbud acyclicity criterion \cite{BE}. Since $A$ and $B$
are maximal Cohen--Macaulay over $R$, an $R$-regular sequence in any
of these ideals is also regular on either wing. The same grade
conditions therefore hold after tensoring with $A$ or $B$, so
$F\otimes_R A$ and $F\otimes_R B$ are acyclic in positive degrees.
Their differentials still have entries in the maximal ideals, so these
resolutions remain minimal. Since $M_B$ is nonzero and has finite
length, Auslander--Buchsbaum gives
$\pdim_R M=\pdim_A M_A=\pdim_B M_B=d$ and $\depth_R M=\depth_A M_A=0$.
The vanishing against $C$ follows from \cref{eq:ring-patching}.
If $\Tor_1^R(M,C)$ also vanished, $F\otimes_R C$ would resolve
$M\otimes_R C$ minimally over $C$ in length~$d$. Since $\depth C=d-1$,
Auslander--Buchsbaum rules this out.
\end{proof}

Thus \cref{eq:torsion-components} applies, and
$\H^0_{\m}(M_A)\neq0$. The linked-wing obstruction gives the
following consequence.

\begin{corollary}\label[corollary]{cor:gorenstein}
Suppose that $A$ and $B$ are linked Gorenstein wings. If an $R$-module
of finite projective dimension has unbalanced support with respect to
these wings, then $R$ is Cohen--Macaulay but not Gorenstein.
\end{corollary}

\begin{proof}
By the New Intersection Theorem reduction in \cref{eq:nit-reduction},
$R$ is Cohen--Macaulay. Suppose that $R$ is Gorenstein. Then the canonical
modules of its maximal Cohen--Macaulay quotients $A$ and $B$ are
\[
 \omega_A\cong\Hom_R(A,R)=J,\qquad
 \omega_B\cong\Hom_R(B,R)=I.
\]
Thus $I$ and $J$ are principal. Proposition~\ref{obs:restriction} gives
$\Tor_1^R(M,A)=0$, contradicting Proposition~\ref{prop:principal}.
\end{proof}

The example in \cite{Ma} has linked Gorenstein wings and a
Cohen--Macaulay, non-Gorenstein ambient ring, as required by
\cref{cor:gorenstein}.


\subsection{Intersection multiplicity on the two wings}

Suppose that $M$ has finite projective dimension.
By \cref{obs:restriction}, its derived restrictions to $A$ and
$B$ are $M_A$ and $M_B$, respectively. Hence
\begin{equation}\label{eq:derived-interface}
 G:=M\lotimes_R C
 \simeq M_A\lotimes_A C
 \simeq M_B\lotimes_B C.
\end{equation}
The perfect $C$-complex $G$ has finite-length homology.

For finitely generated $S$-modules $U,V$ with $\pdim_S U<\infty$ and
$\ell(U\otimes_S V)<\infty$, set
\[
 \chi^S(U,V)=\sum_i(-1)^i\ell\bigl(\Tor_i^S(U,V)\bigr).
\]
By \emph{generalized vanishing} we mean the assertion that
$\chi^S(U,V)=0$ when $\dim U+\dim V<\dim S$,
and by \emph{generalized positivity} the assertion that
$\chi^S(U,V)>0$ when $\dim U+\dim V=\dim S$.
These extend Serre's \cite{Ser} vanishing and positivity assertions to
possibly singular local rings, with finite projective dimension
assumed only for $U$.

\begin{proposition}\label[proposition]{prop:bad-intersection}
Suppose that $M$ has finite projective dimension and satisfies
\cref{eq:unbalanced}. The two Euler characteristics agree:
\begin{equation}\label{eq:chi-matching}
 \chi^A(M_A,C)=\chi^B(M_B,C).
\end{equation}
If $\dim M=1$, then either generalized vanishing fails for
$(M_B,C)$ over $B$, or generalized positivity fails for
$(M_A,C)$ over $A$. If $\dim M\geq2$, then the pair
$(M_A,C)$ violates the dimension inequality over $A$.
\end{proposition}

\begin{proof}
The equality of Euler characteristics follows from
\cref{eq:derived-interface}. Since $\dim M_B=0$, $\dim M_A=\dim M$, and $\dim C=d-1$,
the two dimension sums are $d-1$ and $\dim M+d-1$, respectively.
If $\dim M=1$, generalized vanishing over $B$ and generalized positivity
over $A$ would force the common Euler characteristic to be both zero and
positive. If $\dim M\geq2$, then
\[
 \dim M_A+\dim C>d=\dim A,
\]
so $(M_A,C)$ violates the dimension inequality over $A$.
\end{proof}

\subsection{An example with a regular supporting wing}\label{subsec:variation}

We give a variation of the construction in \cite{Ma}. We glue a regular
threefold $\Spec \, A_0$ and a three-dimensional quadric $\Spec \, B_0$
along a regular surface $\Spec \, C_0$, and write
$R_0=A_0\times_{C_0}B_0$ for the resulting ring. The regular threefold
is the supporting wing. We begin with a finite-length module on the
quadric and use its derived restriction to $C_0$ to choose the module
on $A_0$. After patching their resolutions, we identify what each wing
contributes to the torsion of the resulting $R_0$-module.

\setcounter{constructionstep}{0}
\constructionstep{The collapsed wing}
By \cref{prop:bad-intersection}, we need a module on the quadric whose
Euler characteristic against the common surface is positive. The top
Ext dual of a length-$15$ module from the construction of Dutta,
Hochster, and McLaughlin \cite{DHM} gives Euler characteristic one.
Set
\[
 B_0=k[[x,y,z,w]]/(xw-yz),\qquad P_0=(x,z),\qquad
 C_0=B_0/P_0\cong k[[y,w]].
\]
Let $L^\vee$ be the length-$15$ module denoted by that symbol in
\cite[Section~2.1]{Ma}. Its homology after restriction and the ranks in
its resolution determine the derived restriction. This recalls the
argument in \cite[Remark~7]{Ma}.

\begin{lemma}\label[lemma]{lem:dhm-positive-interface}
The module $M_{B_0}=\Ext_{B_0}^3(L^\vee,B_0)$ has length $15$, finite
projective dimension, and
\begin{equation}\label{eq:dhm-positive-interface}
 M_{B_0}\lotimes_{B_0} C_0\simeq k^6\oplus\shift k^5.
\end{equation}
\end{lemma}

\begin{proof}
Put $G=L^\vee\lotimes_{B_0} C_0$. The calculation in \cite[Section~2.1]{Ma} gives
$\H_0(G)\cong k^5$, $\H_1(G)\cong k^6$, and no other homology.
The minimal $B_0$-free resolution of $L^\vee$ has ranks
$(5,16,17,6)$, starting in degree zero
\cite[Remark~(5.3)]{DHM}.
Tensoring with $C_0$ preserves minimality, so these are also the
Betti numbers of $G$ over $C_0$.

Tensoring the truncation triangle
$\shift k^6\to G\to k^5\xrightarrow{\eta}\shift^2 k^6$ with $k$ in
the derived category gives the exact sequence
\[
 \underbrace{\Tor_2^{C_0}(k,k^5)}_{k^5}
 \xrightarrow{\bar\eta} k^6
 \longrightarrow \H_1(k\lotimes_{C_0}G)
 \longrightarrow \underbrace{\Tor_1^{C_0}(k,k^5)}_{k^{10}}
 \longrightarrow 0 .
\]
Thus $\beta_1^{C_0}(G)=6-\rank(\bar\eta)+10$. Since this Betti
number is $16$, $\bar\eta=0$. The Koszul resolution over
$C_0=k[[y,w]]$ identifies $\Ext_{C_0}^2(k^5,k^6)$ with
$\Hom_k(k^5,k^6)$, taking $\eta$ to $\bar\eta$. Hence $\eta=0$,
so the identity of $k^5$ lifts to a section of $G\to k^5$. The
triangle splits, giving $G\simeq k^5\oplus\shift k^6$.

Since $B_0$ is Gorenstein of dimension three, dualizing the finite free
resolution of $L^\vee$ resolves $M_{B_0}$, and local duality gives
$\ell(M_{B_0})=15$. Derived base change now yields
\[
 M_{B_0}\lotimes_{B_0} C_0
 \simeq\shift^3\RHom_{C_0}(G,C_0)
 \simeq k^6\oplus\shift k^5,
\]
using $\RHom_{C_0}(k,C_0)\simeq\shift^{-2}k$.
\end{proof}

\constructionstep{The supporting wing}
Set $A_0=C_0[[u]]=k[[y,w,u]]$, so the common surface is cut out by
$u=0$. The $u$-axis is cut out by $y=w=0$ and has coordinate ring
$N_0=A_0/(y,w)\cong k[[u]]$. It meets the common surface only at
the closed point. Its derived restriction to $C_0$ is $k$, and each
copy of the residue field at the origin contributes
$k\lotimes_{A_0}C_0\simeq k\oplus\shift k$. Thus we take
$M_{A_0}=N_0\oplus k^5$. It has finite projective dimension since
$A_0$ is regular, and
\[
 M_{A_0}\lotimes_{A_0} C_0\simeq k^6\oplus\shift k^5
 \simeq M_{B_0}\lotimes_{B_0} C_0.
\]

\constructionstep{The patching}
We use the patching method in \cite[Section~2.3]{Ma}.
The minimal finite free resolutions of $M_{A_0}$ and $M_{B_0}$
restrict to isomorphic complexes over $C_0$, since they remain minimal
and represent the same derived object. The degreewise changes of basis
giving this isomorphism lift to invertible matrices over $A_0$ because
$A_0\to C_0$ is a local surjection. In these bases the differentials
agree over $C_0$, so their corresponding entries pair to give matrices
over the fiber product
\[
 R_0=A_0\times_{C_0} B_0
 \cong k[[x,y,z,w,u]]/(xw-yz,ux,uz).
\]
These matrices define a finite free complex $F_0$ over $R_0$ whose
restrictions are the chosen resolutions. Tensoring \cref{eq:ring-patching} with $F_0$
gives the patching triangle. The two wing resolutions have homology
only in degree zero, and their common restriction to $C_0$ has homology
only in degrees zero and one. The homology sequence of the triangle
therefore gives $\H_i(F_0)=0$ for $i>0$.
Thus $F_0$ resolves $M_0:=\H_0(F_0)$ with the prescribed restrictions.

\constructionstep{The torsion--spine structure}
Write $\m_0$ for the maximal ideal of $R_0$ and
$T_0=\H^0_{\m_0}(M_0)$. The spine $N_0=M_0/T_0=A_0/(y,w)$ is the cyclic
$R_0$-module
$R_0/(x,y,z,w)\cong k[[u]]$, so the support of $M_0$ is the $u$-axis.
Under the two base changes in
\cref{sec:dimension-inequality-pair}, its annihilator will cut out
the two quadric supports. The torsion--spine sequence of
\cref{prop:structure} is
\begin{equation}\label{eq:dhm-spine}
 0\longrightarrow T_0\longrightarrow M_0\longrightarrow k[[u]]
 \longrightarrow0,
\end{equation}
with
\begin{equation}\label{eq:dhm-torsion}
 0\longrightarrow P_0\otimes_{B_0} M_{B_0}\longrightarrow T_0
 \longrightarrow k^5\longrightarrow0.
\end{equation}
Thus the five copies of $k$ on the regular wing are a quotient of
$T_0$, with kernel contributed by the quadric.
Tensoring $0\to P_0\to B_0\to C_0\to0$ with $M_{B_0}$ and using
\cref{lem:dhm-positive-interface} gives the exact sequence
\[
 0\longrightarrow k^5\longrightarrow P_0\otimes_{B_0} M_{B_0}
 \longrightarrow M_{B_0}\longrightarrow k^6\longrightarrow0.
\]
Taking lengths here and in \cref{eq:dhm-torsion} gives
\[
 \ell(P_0\otimes_{B_0} M_{B_0})=5+15-6=14,\qquad \ell(T_0)=14+5=19.
\]

Since $(u)\cong A_0$ and $\Tor_1^{R_0}(M_0,B_0)=0$, multiplication
by $u$ identifies $M_0\otimes_{R_0}A_0$ with $uM_0$. The image of
$T_0$ on the $A_0$-wing is $k^5$ by \eqref{eq:dhm-torsion}, so
$uT_0\cong k^5$, whence $\ll_{R_0}(T_0)\geq2$ and $u^2T_0=0$.

The construction in \cite{Ma} uses the same ambient ring, with the
hypersurface as the supporting wing and the regular ring as the collapsed
wing. Its common derived restriction is $k^6\oplus\shift k^6$ and has
Euler characteristic zero.
Here the wing roles are reversed, and
\[
 \chi^{A_0}(M_{A_0},C_0)=\chi^{B_0}(M_{B_0},C_0)=6-5=1.
\]
Thus the alternative in \cref{prop:bad-intersection} is realized by
ordinary positivity on the regular $A_0$-wing and failure of generalized
vanishing on the collapsed $B_0$-wing. The quotient $B_0$ has depth three,
so finite projective dimension over $R_0$ would make it free by
Auslander--Buchsbaum. The quotient has nonzero annihilator $(u)$, so it cannot be free.

\begin{remark}\label[remark]{rem:lifting}
The two restrictions give contrasting lifting behavior. By
\cref{obs:restriction}, $M_0$ is a homological lift of $M_{B_0}$ along
$R_0\onto B_0$: one has $M_0\lotimes_{R_0} B_0\simeq M_{B_0}$. But
$\dim R_0=\dim B_0$ and $\dim M_0>\dim M_{B_0}$, so this lift does not
satisfy the dimension equality for a Serre lift
\cite[Definition~1.4]{KS}. Along $R_0\onto A_0$, the same module is
both a homological lift and a Serre lift, since
$\dim R_0=\dim A_0$ and $\dim M_0=\dim M_{A_0}$.

Over a regular source, every homological lift is a Serre lift
\cite[Proposition~1.6]{KS}. Our examples show that this can fail for
a particular lift over a singular source. They do not give a module
with no Serre lift: since $\dim R_0=\dim B_0$, the module $M_{B_0}$ itself,
viewed as an $R_0$-module, is a Serre lift along $R_0\onto B_0$.
\end{remark}

\begin{question}\label[question]{q:lifting}
Does there exist a surjection of Noetherian local rings $R\onto B$
and a nonzero finitely generated $B$-module $V$ that admits a
homological lift to $R$, but for which no finitely generated
$R$-module $L$ satisfies
\[
 L\otimes_R B\cong V,\qquad
 \dim R-\dim L=\dim B-\dim V?
\]
\end{question}


\section{Strong multiplicity conjectures}
\label{sec:dimension-inequality-pair}

\subsection{The dimension inequality conjecture}
In this section, we use modules of unbalanced support to answer negatively
the strengthened version of the dimension inequality conjecture in which
both modules have finite projective dimension \cite[Question~3]{Ma}.

We base change the seed module $M_0$ of Section~\ref{subsec:variation}
to a five-dimensional fiber product. Its spine becomes the coordinate
ring of a three-dimensional quadric, and the image of the seed torsion
is a two-dimensional correction term. Interchanging the wings gives a
second module of finite projective dimension, supported on a quadric
that meets the first only at the closed point.

\setcounter{constructionstep}{0}
\constructionstep{Two quadrics meeting at a point}
Let $H_1$ and $H_2$ be two copies of the quadric $B_0$ of
\cref{subsec:variation}. Over the same field $k$, set
\[
 f_i=x_iw_i-y_iz_i,\qquad
 H_i=k[[x_i,y_i,z_i,w_i]]/(f_i),\qquad P_i=(x_i,z_i)H_i .
\]
Thus $H_i/P_i\cong k[[y_i,w_i]]$. We glue the five-dimensional
quadric hypersurfaces
\[
 A=H_1[[y_2,w_2]],\qquad B=H_2[[y_1,w_1]]
\]
along $C=k[[y_1,w_1,y_2,w_2]]$. The maps $A\to C$ and $B\to C$
kill $P_1$ and $P_2$, respectively. Their fiber product is
\[
 R=A\times_C B\cong
 \frac{k[[x_1,y_1,z_1,w_1,x_2,y_2,z_2,w_2]]}
      {(f_1,f_2)+(x_1,z_1)(x_2,z_2)} .
\]
Indeed, the denominator is the intersection of the prime ideals
$(f_1,x_2,z_2)$ and $(x_1,z_1,f_2)$. In particular, $R$ is complete
and reduced. Since the two wings have dimension five and their common
quotient is regular of dimension four, the depth lemma applied to
\begin{equation}\label{eq:fivefold-patching}
 0\longrightarrow R\longrightarrow A\oplus B
 \xrightarrow{\ (a,b)\mapsto\bar a-\bar b\ }C\longrightarrow0
\end{equation}
gives $\depth R=\dim R=5$. Write $\mathfrak n_i=(x_i,y_i,z_i,w_i)R$.
The quadrics $H_1=R/\mathfrak n_2$ and $H_2=R/\mathfrak n_1$ lie
on the first and second wings, respectively.

\begin{lemma}\label[lemma]{lem:fivefold-quadrics}
The closed subschemes $\Spec \, H_1$ and $\Spec \, H_2$ intersect in
the reduced closed point of $\Spec \, R$. Both $H_1$ and $H_2$ have
infinite projective dimension over $R$.
\end{lemma}

\begin{proof}
The intersection has coordinate ring
$H_1\otimes_R H_2=R/(\mathfrak n_1+\mathfrak n_2)=k$.
If $H_1$ had finite projective dimension, Auslander--Buchsbaum would
give $\pdim_R H_1=5-3=2$, while the New Intersection Theorem applied
to $H_1$ and $H_2$ gives $\pdim_R H_1\geq\dim H_2=3$.
The same argument applies to $H_2$.
\end{proof}

We now adjoin a two-dimensional correction term to each $H_i$ to
obtain a module of finite projective dimension with the same support.

\constructionstep{Transporting the spine}
Recall that $R_0=k[[x,y,z,w,u]]/(xw-yz,ux,uz)$. Since
$x_1(x_2,z_2)=0$ in $R$, there is a map
\[
 \phi\colon R_0\longrightarrow R,\qquad
 (x,y,z,w,u)\longmapsto(x_2,y_2,z_2,w_2,x_1).
\]
Set $U_1=M_0\otimes_{R_0}{}^{\phi}\!R$. The spine
$k[[u]]=R_0/(x,y,z,w)$ extends to $R/\mathfrak n_2=H_1$.
Right exactness applied to \cref{eq:dhm-spine} gives
\begin{equation}\label{eq:fivefold-spine}
 0\longrightarrow L_1\longrightarrow U_1\longrightarrow H_1
 \longrightarrow0,\qquad
 L_1=\operatorname{im}\bigl(T_0\otimes_{R_0}{}^{\phi}\!R\to U_1\bigr).
\end{equation}
A power of $(\mathfrak n_2,x_1)$ kills $L_1$ because $T_0$ has finite
length. It follows at once that $\Supp_R U_1=\Spec \, H_1$. On the
first wing this is the whole three-dimensional quadric. On the second
it is the plane cut out on $H_1$ by $(x_1,z_1)$, with coordinate ring
$k[[y_1,w_1]]$. Thus the dimensions $(1,0)$ of the two restrictions
of $M_0$ become $(3,2)$ for $U_1$.

\constructionstep{Transporting the resolution}
Let $F_0$ be the minimal free resolution of $M_0$ in
\cref{subsec:variation} and set $F=F_0\otimes_{R_0}{}^{\phi}\!R$.
We check exactness on the two wings and on their intersection. Write
$F_A,F_B,F_C$ for the three restrictions. On $A$ the map factors
through $A_0$, with $(y,w,u)\mapsto(y_2,w_2,x_1)$. Thus
\[
 F_A\cong K_A(y_2,w_2)\oplus K_A(y_2,w_2,x_1)^{\oplus5}.
\]
The sequences are regular, since $y_2,w_2$ are the adjoined variables
and $x_1$ is regular on $H_1$. Hence $F_A$ resolves
$H_1\oplus(H_1/(x_1))^{\oplus5}$. On $B$ the map factors through
the flat extension $B_0\cong H_2\to H_2[[y_1,w_1]]=B$.
Consequently $F_B$ resolves $V_2[[y_1,w_1]]$, where $V_i$ denotes
the copy of $M_{B_0}$ on $H_i$.

The two derived restrictions on $C$ agree by extension of their
common restriction on $C_0$. Indeed, both composites induce the flat
map $C_0\to C=C_0[[y_1,w_1]]$ sending $(y,w)$ to $(y_2,w_2)$.
It follows from \cref{eq:dhm-positive-interface} that
$\H_i(F_C)=0$ for $i\geq2$. Tensoring \cref{eq:fivefold-patching}
degreewise with $F$ gives
\[
 0\longrightarrow F\longrightarrow F_A\oplus F_B
 \longrightarrow F_C\longrightarrow0.
\]
The homology sequence gives $\H_i(F)=0$ for $i>0$. Hence $F$
resolves $U_1$. It remains minimal of length three since $\phi$ is local.

\constructionstep{Interchanging the wings}
The involution $\sigma$ of $R$ exchanges $(x_1,y_1,z_1,w_1)$ and
$(x_2,y_2,z_2,w_2)$. Put $\psi=\sigma\circ\phi$ and
$U_2=M_0\otimes_{R_0}{}^{\psi}\!R$. The $\sigma$-linear isomorphism
$m\otimes r\mapsto m\otimes\sigma(r)$ from $U_1$ to $U_2$ is
$k$-linear. It transports \cref{eq:fivefold-spine} to
\[
 0\longrightarrow L_2\longrightarrow U_2\longrightarrow H_2
 \longrightarrow0 .
\]
In particular, the restrictions of $U_2$ have dimensions two on $A$
and three on $B$.

\begin{theorem}\label[theorem]{thm:fivefold-pair}
The modules $U_1,U_2$ have projective dimension three over $R$ and
$\Supp_R U_i=\Spec \, H_i$ for $i=1,2$. Their supports meet only at
the closed point, although $\dim U_1+\dim U_2=3+3>5=\dim R$.
\end{theorem}

\begin{proof}
The support and resolution calculations give the assertions for $U_1$,
and $\sigma$ gives them for $U_2$. The two supports therefore meet only at $\m$ by
\cref{lem:fivefold-quadrics}.
\end{proof}

Here $\pdim_R U_i=\dim U_j$ for $j\ne i$, so equality holds in the
New Intersection Theorem. Auslander--Buchsbaum gives $\depth_R U_i=2$.
The seed torsion gives a more precise description of the correction terms.

\Needspace{10\baselineskip}
\begin{proposition}\label[proposition]{prop:fivefold-correction}
The correction term $L_i$ is the $x_i$-power torsion of $U_i$.
More precisely,
\[
 L_i=\H^0_{(x_i)}(U_i)=(0:_{U_i}x_i^2).
\]
Let $E_i=P_i\otimes_{H_i}V_i$, the copy on $H_i$ of the left term
in the seed torsion sequence \eqref{eq:dhm-torsion}. Then there are
exact sequences
\[
\begin{aligned}
 0&\longrightarrow E_2[[y_1,w_1]]
   \longrightarrow L_1
   \longrightarrow (H_1/(x_1))^{\oplus5}
   \longrightarrow0,\\
 0&\longrightarrow E_1[[y_2,w_2]]
   \longrightarrow L_2
   \longrightarrow (H_2/(x_2))^{\oplus5}
   \longrightarrow0.
\end{aligned}
\]
In particular, each $L_i$ is Cohen--Macaulay of dimension two and has
infinite projective dimension over $R$.
\end{proposition}

\begin{proof}
By \cref{subsec:variation}, $u^2T_0=0$, so $x_1^2L_1=0$.
Since $x_1$ is regular on $U_1/L_1=H_1$, this gives the torsion
identification.

Put $I=\ker(R\to A)=P_2B$. Since $F_A$ resolves $U_1\otimes_R A$,
we have $\Tor_1^R(U_1,A)=0$. Hence
\[
 IU_1\cong I\otimes_B(U_1\otimes_R B)\cong E_2[[y_1,w_1]].
\]
The induced surjection $U_1/IU_1=U_1\otimes_R A\cong
H_1\oplus(H_1/(x_1))^{\oplus5}\to H_1$
kills the second summand because $x_1$ is regular on $H_1$. On the
first summand it is multiplication by a unit. Its kernel is therefore
$L_1/IU_1\cong(H_1/(x_1))^{\oplus5}$, giving the first sequence.
The involution $\sigma$ gives the second.

Both outer terms are Cohen--Macaulay of dimension two, so the depth
lemma gives the same conclusion for $L_i$.

Finally, $U_i$ has finite projective dimension and $H_i$ has infinite
projective dimension by \cref{lem:fivefold-quadrics}, so $L_i$ has
infinite projective dimension.
\end{proof}

The sequence also locates the two pieces of $L_1$. The submodule
$E_2[[y_1,w_1]]$ is supported on the plane $x_1=z_1=0$ in
$\Spec \, H_1$, where $H_1$ meets the second wing. The quotient is
supported on the divisor $x_1=0$. In $H_1$ this divisor is the union of
that plane and the plane $x_1=y_1=0$. Thus the whole correction term
has support $\V_{H_1}(x_1)$. The involution gives the same description
for $L_2$.

\begin{example}
By \cref{subsec:variation}, $\ell(E_2)=14$, so $E_2[[y_1,w_1]]$
has multiplicity $14$. Since
$H_1/(x_1)\cong k[[y_1,z_1,w_1]]/(y_1z_1)$ has multiplicity two,
additivity in \cref{prop:fivefold-correction} and symmetry give
\[
 e_{\m}(L_1)=e_{\m}(L_2)=14+5\cdot2=24.
\]
Here $e_{\m}(-)$ denotes the Hilbert--Samuel multiplicity at $\m$.
\end{example}

\subsection{Counterexamples to the strong vanishing conjecture}
\label{subsec:intersection-multiplicities}

In this section, we construct counterexamples to the strong vanishing conjecture. 

\begin{thm}\label[theorem]{thm:strong-vanishing}
Let $H$ be a $d$-dimensional complete Cohen--Macaulay local ring with
coefficient field $k$, and let $C_0$ be a Cohen--Macaulay quotient of
dimension $d-1$. Suppose that $L$ is a finite-length $H$-module of
finite projective dimension with \[ r=\chi^H(L,C_0)\ne 0.\]
Then there is a complete Cohen--Macaulay local ring $(S,\mathfrak n)$
of dimension $2d-1$ and two Cohen--Macaulay $S$-modules $M_1,M_2$,
each of dimension $d-1$ and projective dimension $d$, such that $\Supp_S M_1\cap\Supp_S M_2=\{\mathfrak n\}$, and
\[
 \dim M_1 + \dim M_2 = 2d-2 < 2d-1 = \dim S,
  \qquad
  \chi^S(M_1,M_2)=-r^2.
\]
\end{thm}

Starting from $H$, we construct a ring $D$ with a quotient $C$. We
glue two copies of $D$ along $C$ and use the defining exact sequence of
the fiber product to compute the intersection multiplicity.

\begin{proof}
The proof is in three steps. 

\setcounter{constructionstep}{0}
\constructionstep{Construction of the wings}
The kernel of $H\to C_0$ has height one, so it contains an $H$-regular
element $t$. For $i=1,2$, let $(H_i,C_i,L_i,t_i)$ be a copy of
$(H,C_0,L,t)$, and write $\m_i$ for the maximal ideal of $H_i$. Set
\[
  E=H_1\widehat\otimes_k H_2,\qquad
  D=E/(t_1-t_2),\qquad
  C=C_1\widehat\otimes_k C_2.
\]
The quotient map $E\to C$ factors through $D$. For $\{i,j\}=\{1,2\}$,
the completed tensor products over $k$ give flat local maps $H_i\to E$
and $C_i\to C$ with closed fibres $H_j$ and $C_j$, respectively
\cite[Chap.~V]{Ser}. Since $t_1-t_2$ restricts to the nonzerodivisor
$\pm t_j$ on the closed fibre of $H_i\to E$, the map $H_i\to D$ is
flat by \cite[Lemma~10.99.2, Tag~00MF]{Stacks}, with closed fibre
$H_j/(t_j)$. Thus $D$ and $C$ are Cohen--Macaulay of dimensions
$2d-1$ and $2d-2$, respectively.

Set $N_i=L_i\otimes_{H_i}D$. Flatness carries a minimal free
resolution of $L_i$ to one of $N_i$, of length $d$ by
Auslander--Buchsbaum. Since $L_i$ has finite length,
\[
  \Supp_D N_i=\V(\m_iD),\qquad D/\m_iD\cong H_j/(t_j).
\]
These supports have dimension $d-1$ and meet only at the closed point,
since $D/(\m_1D+\m_2D)=k$.

\constructionstep{Gluing the wings and constructing the two modules}
Let $S=D\times_C D$, with $D$-algebra structure given by the diagonal
\[
  D\longrightarrow S,\qquad d\longmapsto(d,d).
\]
The exact sequence of $D$-modules
\[
  0\longrightarrow S\longrightarrow D\oplus D
   \xrightarrow{(a,b)\mapsto\bar a-\bar b} C\longrightarrow 0
\]
makes $S$ a finite maximal Cohen--Macaulay $D$-module by the depth
lemma. Thus $S$ is a complete Cohen--Macaulay local ring of dimension
$2d-1$. As in the proof of Proposition~\ref{obs:restriction}, a
minimal $D$-free resolution of $N_i$ remains a minimal free
resolution after extension to $S$. Hence $N_i\otimes_D S$ has
projective dimension $d$ and depth $d-1$. Its support is the inverse
image of $\V(\m_iD)$ under the finite surjection $\Spec S\to\Spec D$.
These inverse images have dimension $d-1$ and meet only at the closed
point of $S$. We take the two Cohen--Macaulay modules
\[
  M_i=N_i\otimes_D S \cong L_i\otimes_{H_i}S
     \simeq L_i\lotimes_{H_i}S \simeq N_i\lotimes_D S \qquad (i=1,2).
\]

\constructionstep{Computing the intersection multiplicity}
Let $F_i\to L_i$ be a minimal $H_i$-free resolution. As in Serre's
power series case \cite[p.~105]{Ser}, the total complex
$F=\operatorname{Tot}(F_1\widehat\otimes_k F_2)$ is an $E$-free
resolution of $T=L_1\otimes_k L_2$. Extending scalars gives
\[
  F\otimes_E S
  \cong (F_1\otimes_{H_1}S)\otimes_S(F_2\otimes_{H_2}S)
  \simeq M_1\lotimes_S M_2.
\]
The residue fields agree, so $\chi^E(T,S)=\chi^S(M_1,M_2)$.
The same calculation over $D$ gives $\chi^D(N_1,N_2)=\chi^E(T,D)$.
The lemma in Step~I also gives that $t_1-t_2$ is $E$-regular. Hence
$\chi^E(T,D)$ is the Euler characteristic of
$[T\xrightarrow{t_1-t_2}T]$, which is zero since $T$ has finite length. Setting $G_i=F_i\otimes_{H_i}C_i$, 
K\"unneth formula \cite[Tag~0H7Z]{Stacks} over $k$ gives

\[
  \H_n(F\otimes_E C)
  \cong \H_n\bigl(\operatorname{Tot}(G_1\widehat\otimes_k G_2)\bigr)
  \cong \bigoplus_{p+q=n}\H_p(G_1)\otimes_k\H_q(G_2).
\]
 Taking Euler characteristics gives
\[
  \chi^E(T,C)=\chi^{H_1}(L_1,C_1)\,\chi^{H_2}(L_2,C_2)=r^2.
\]
Finally, tensor the exact sequence in Step~II, viewed over $E$, with
$F$. Additivity of Euler characteristics gives
\[
  \chi^S(M_1,M_2)=2\chi^E(T,D)-\chi^E(T,C)=-r^2. \qedhere
\]
\end{proof}

Examples satisfying the condition of the theorem are found in
\cite{DHM,RS}. For example, starting with the Dutta--Hochster--McLaughlin example
and taking $t=x$ in the notation of Section~\ref{subsec:variation}
gives a five-dimensional complete intersection $D$. Gluing $\Spec\, D$
to itself along the regular fourfold $\Spec\, C$ produces a reduced
Cohen--Macaulay local ring of dimension five over which strong vanishing fails.

The following question is motivated by the vanishing theorem for complete
intersections, proved independently by Roberts~\cite{Ro4} and
Gillet--Soul\'e~\cite{GS}. Vanishing also holds for isolated
singularities, as Roberts and Srinivas show in~\cite[Theorem~2]{RS}.

\begin{question}
Let $(R, \m)$ be a Noetherian local ring, and let $M$ and $N$ be nonzero
finitely generated $R$-modules of finite projective dimension with
${\Supp_R M\cap\Supp_R N=\{\m\}}$. Can the dimension
inequality or positivity fail if $R$ is a complete intersection or has an isolated singularity?

\end{question}


\section{A Gorenstein construction}
\label{sec:gorenstein}

We will now return to the study of modules of finite projective dimension and unbalanced support over Gorenstein ambient rings.
In \cref{subsec:variation} both wings are Gorenstein, so
\cref{cor:gorenstein} forces the ambient ring to be non-Gorenstein.
Here we construct an example where one wing is a quadric hypersurface and another is of type three. On this single Gorenstein ring, 
we construct two modules of finite projective dimension and unbalanced support collapsing on either wing. 

\begin{theorem}\label[theorem]{thm:roberts-ma}
For every field $k$ of characteristic different from two, there is a
three-dimensional complete reduced Gorenstein local ring $R$ with residue
field $k$ and exactly two irreducible components. Their coordinate rings
are linked wings $A$ and $B$, where $A$ is a hypersurface and $B$ is
Cohen--Macaulay of type three. There exist one-dimensional $R$-modules $M$ and $M'$
of finite projective dimension. The module $M$ has supporting wing $A$ and
collapsed wing $B$, and $M'$ has supporting wing $B$ and collapsed wing $A$.
\end{theorem}

The spine meets the common quotient in a module of positive length.
The torsion--spine balance determines which wing must account for this
length. Roberts' module and its top Ext dual have lengths $31$ and $29$.
Their difference prescribes a spine of length two on the intersection.
We first prove the balance, then use it to choose the modules to patch
in the two cases.

\subsection{Numerical torsion--spine balance}

Work in the linked two-wing setting of \cref{sec:setup}. Throughout
this subsection, $A$ denotes the supporting wing and $B$ the collapsed
wing of the module under consideration. We begin the comparison over
a Cohen--Macaulay ambient ring.

\begin{proposition}\label[proposition]{prop:numerical-torsion-spine-balance}
Let $M$ have finite projective dimension over $R$ and satisfy
\cref{eq:unbalanced}, and set $N=M/\H^0_{\m}(M)$. Then
\begin{equation}\label{eq:truncated-tor-balance}
\begin{aligned}
 &\sum_{i=0}^{2}(-1)^i\ell\bigl(\Tor_i^A(N,C)\bigr)\\
 &\qquad=
 \bigl[\ell(M_B\otimes_B C)-\ell\bigl(\Tor_1^B(M_B,C)\bigr)\bigr]\\
 &\qquad\quad-
 \bigl[\ell(\H^0_{\m}(M_A)\otimes_A C)
       -\ell\bigl(\Tor_1^A(\H^0_{\m}(M_A),C)\bigr)\bigr].
\end{aligned}
\end{equation}
\end{proposition}

\begin{proof}
By \cref{prop:structure}, there is an exact sequence
\[
 0\longrightarrow\H^0_{\m}(M_A)\longrightarrow M_A
 \longrightarrow N\longrightarrow0.
\]
Since $\Tor_i^A(M_A,C)=0$ for $i\geq2$ by
Proposition~\ref{obs:restriction} and \cref{eq:derived-interface}, the long
exact Tor sequence, with the terms involving $M_A$ identified through
\cref{eq:derived-interface}, contains
\[
\begin{aligned}
 0&\longrightarrow\Tor_2^A(N,C)
 \longrightarrow\Tor_1^A(\H^0_{\m}(M_A),C)
 \longrightarrow\Tor_1^B(M_B,C)\\
 &\longrightarrow\Tor_1^A(N,C)
 \longrightarrow\H^0_{\m}(M_A)\otimes_A C
 \longrightarrow M_B\otimes_B C\\
 &\longrightarrow N\otimes_A C\longrightarrow0.
\end{aligned}
\]
Every term has finite length, since the supports meet $C$ only at
the closed point. Taking alternating lengths proves the formula.
\end{proof}

\begin{corollary}\label[corollary]{cor:perfect-spine-balance}
Under the hypotheses of \cref{prop:numerical-torsion-spine-balance},
$\pdim_A N<\infty$ if and only if $\pdim_A\H^0_{\m}(M_A)<\infty$.
In this case $\depth_A N=1$ and $\Tor_i^A(N,C)=0$ for $i>0$, so
the left-hand side of \cref{eq:truncated-tor-balance} is
$\ell(N\otimes_A C)>0$.
\end{corollary}

\begin{proof}
The equivalence follows from $0\to\H^0_{\m}(M_A)\to M_A\to N\to0$
since $\pdim_A M_A<\infty$. If $\pdim_A N<\infty$, then
$\depth_A N>0$ gives $\pdim_A N\leq d-1=\depth C$. Tensor a minimal
$A$-free resolution of $N$ with $C$. Its homology has finite length,
so the acyclicity lemma gives vanishing in positive degrees. The
resulting minimal resolution of the nonzero finite-length $C$-module
$N\otimes_A C$ has length $d-1$ by Auslander--Buchsbaum. Thus
$\pdim_A N=d-1$ and $\depth_A N=1$.
\end{proof}

Assume for the rest of this subsection that $R$ is Gorenstein of
dimension $d$.

\begin{observation}\label[observation]{obs:canonical-wings}
The linked ideals $J$ and $I$ are canonical ideals of $A$ and $B$,
respectively, with common quotient $C$. Indeed, canonical duality
\cite[Lecture~11]{24H} and linkage give
\[
 \omega_A\cong\Hom_R(A,R)=J,\qquad
 \omega_B\cong\Hom_R(B,R)=I.
\]
Since $I\cap J=0$, the projections embed $J$ in $A$ and $I$ in $B$.
The quotients by these ideals are both $R/(I+J)=C$. In particular,
$\mu_R(J)=\operatorname{type}(A)$ and $\mu_R(I)=\operatorname{type}(B)$.
\end{observation}

\begin{lemma}\label[lemma]{lem:canonical-length}
Let $(S,\mathfrak n)$ be a Cohen--Macaulay local ring of dimension $d$
with a canonical ideal $\omega_S\subseteq S$. For every finite-length
$S$-module $L$, there is an isomorphism
\[
 \Ext_S^d(L,S)\cong(L\otimes_S\omega_S)^\vee,
\]
where $(-)^\vee$ denotes Matlis duality. Consequently, these two modules have the same length and Loewy
length, and
\begin{equation}\label{eq:canonical-length}
\begin{aligned}
 &\ell(L\otimes_S S/\omega_S)
 -\ell\bigl(\Tor_1^S(L,S/\omega_S)\bigr)\\
 &\qquad=
 \ell(L)-\ell(L\otimes_S\omega_S)
 =\ell(L)-\ell\bigl(\Ext_S^d(L,S)\bigr).
\end{aligned}
\end{equation}
If $L$ is nonzero and has finite projective dimension, then
$\Ext_S^d(L,S)$ is also nonzero and has finite projective dimension.
\end{lemma}

\begin{proof}
Since $L$ has finite length and $S$ is Cohen--Macaulay, Grothendieck
duality gives the first two isomorphisms below; see
\cite[Lecture~11]{24H}.
\[
\begin{aligned}
 \Ext_S^d(L,S)
 &\cong\Hom_S(L,\H_{\mathfrak n}^d(S))\\
 &\cong\Hom_S(L,\omega_S^\vee)\\
 &\cong\Hom_S(L\otimes_S\omega_S,E_S(k))\\
 &=(L\otimes_S\omega_S)^\vee.
\end{aligned}
\]
Here $E_S(k)$ is the injective hull of the residue field of $S$.
Matlis duality preserves length and annihilators, giving the asserted
equalities for length and Loewy length.

Tensoring $0\to\omega_S\to S\to S/\omega_S\to0$ with $L$ gives
\[
 0\longrightarrow\Tor_1^S(L,S/\omega_S)
 \longrightarrow L\otimes_S\omega_S
 \longrightarrow L
 \longrightarrow L\otimes_S S/\omega_S
 \longrightarrow0.
\]
Taking alternating lengths proves \cref{eq:canonical-length}.

If $L\neq0$, then $L\otimes_S\omega_S\neq0$, so
$\Ext_S^d(L,S)\neq0$. If, moreover, $\pdim_S L<\infty$, then
Auslander--Buchsbaum gives $\pdim_S L=d$. The dual of a finite free
resolution of $L$ has cohomology only in degree $d$, and therefore
gives a finite free resolution of $\Ext_S^d(L,S)$.
\end{proof}

Canonical duality rewrites the balance as a difference of length
asymmetries on the two wings.

\begin{theorem}\label[theorem]{thm:torsion-spine-balance}
Let $M$ have finite projective dimension over $R$ and satisfy
\cref{eq:unbalanced}, and set $N=M/\H^0_{\m}(M)$. Then
\begin{equation}\label{eq:torsion-spine-balance}
\begin{aligned}
 &\ell(N\otimes_A C)
 -\ell\bigl(\Tor_1^A(N,C)\bigr)
 +\ell\bigl(\Tor_2^A(N,C)\bigr)\\
 &\qquad=
 \bigl[\ell(M_B)-\ell\bigl(\Ext_B^d(M_B,B)\bigr)\bigr]\\
 &\qquad\quad
 -\bigl[\ell\bigl(\H^0_{\m}(M_A)\bigr)
 -\ell\bigl(\Ext_A^d(\H^0_{\m}(M_A),A)\bigr)\bigr].
\end{aligned}
\end{equation}
\end{theorem}

\begin{proof}[Proof of \cref{thm:torsion-spine-balance}]
Apply \cref{lem:canonical-length} over $B$ to $M_B$ and over $A$
to $\H^0_{\m}(M_A)$. Since $C$ is the quotient by a canonical ideal on
each wing, these identify the two differences in
\cref{eq:truncated-tor-balance} with the corresponding length
asymmetries. The result follows from
\cref{prop:numerical-torsion-spine-balance}.
\end{proof}

No finite-projective-dimension hypothesis on $N$, or equivalently on
$\H^0_{\m}(M_A)$, is required.

\begin{corollary}\label[corollary]{prop:length-obstruction}
Under the hypotheses of \cref{thm:torsion-spine-balance}, suppose
that $A$ is Gorenstein, and write $J=(g)$. Then
\begin{equation}\label{eq:length-asymmetry}
 \ell(N/gN)
 =\bigl[\ell(M_B)-\ell\bigl(\Ext_B^d(M_B,B)\bigr)\bigr]>0.
\end{equation}
\end{corollary}

\begin{proof}
Since $g$ is a nonzerodivisor on both $A$ and $N$, the higher Tor
terms on the left-hand side of \cref{eq:torsion-spine-balance} vanish.
Local duality over $A$ makes the supporting-wing length difference
zero, so \cref{thm:torsion-spine-balance} gives the formula; positivity
follows from Nakayama's lemma.
\end{proof}

If $B$ were also Gorenstein, its length difference would vanish,
recovering \cref{cor:gorenstein}. A construction with $R$ and $A$
Gorenstein must therefore supply a strict length asymmetry over $B$.

\begin{corollary}\label[corollary]{cor:supporting-length-obstruction}
Under the hypotheses of \cref{thm:torsion-spine-balance}, suppose that
the collapsed wing $B$ is Gorenstein and $\pdim_A N<\infty$. Then
\begin{equation}\label{eq:supporting-length-asymmetry}
 \ell(N\otimes_A C)
 =\ell\bigl(\Ext_A^d(\H^0_{\m}(M_A),A)\bigr)
  -\ell\bigl(\H^0_{\m}(M_A)\bigr)>0.
\end{equation}
\end{corollary}

\begin{proof}
Local duality over $B$ gives
$\ell(\Ext_B^d(M_B,B))=\ell(M_B)$. The higher Tor terms of the spine
vanish by \cref{cor:perfect-spine-balance}. Substitution in
\cref{eq:torsion-spine-balance} gives the equality, and
$N\otimes_A C\neq0$ by Nakayama's lemma.
\end{proof}

\subsection{Construction}
\label{sec:gorenstein-example}

Let $k$ be a field of characteristic different from two, and set
\[
 H=k[[x,y,z,w]]/(xw-yz),\qquad
 P=(x,z)H,\qquad Q=(x,y)H,\qquad C_0=H/P.
\]

\setcounter{constructionstep}{0}
\constructionstep{Roberts' double cover and its canonical ideal}
The ring $H$ is a normal Gorenstein domain and $Q$ is rank-one
reflexive. Roberts' extension \cite[pp.~36--38]{Ro2} is the double
cover obtained by adjoining a square root of $z/x$ with coefficients
in $Q$. On the locus where $Q$ is invertible, this is the usual
double-cover construction from a line bundle and a section of its
inverse square \cite[Section~3.5]{EV}; the reflexive ideal extends
the multiplication across the non-Cartier locus.
Write $K=\operatorname{Frac}(H)$ and choose $t$ with
$t^2=z/x$. The valuation of $z/x$ at $Q$ is $-1$, so $t\notin K$.
The algebra
\[
 B=H\oplus Qt
\]
is a finite local domain over $H$, with
\[
 (xt)^2=xz,\qquad (xt)(yt)=yz,\qquad (yt)^2=yw.
\]
Indeed, modulo $\m_H B$, the images of $xt$ and $yt$ generate a
square-zero ideal with quotient $k$. The exact sequence
$0\to Q\to H\to H/Q\to0$ gives $\depth_H Q=3$, so $B$ is
Cohen--Macaulay of dimension three.

\begin{proposition}\label[proposition]{prop:roberts-canonical}
For $a=xt$, the ideal
\[
 W=(x,z,a)B=P\oplus Ha
\]
is a canonical ideal of $B$, and $\operatorname{type}(B)=3$.
Writing $b=yt$, one has
\[
 C=B/W\cong\frac{k[[y,w,b]]}{(b^2-yw)}.
\]
In particular, $C$ is Gorenstein and free of rank two over $C_0$.
\end{proposition}

\begin{proof}
Let $\lambda:K(t)\to K$ extract the coefficient of $t$. The
nondegenerate pairing $(v,c)\mapsto\lambda(vc)$ identifies $K(t)$
with its $K$-dual. Since
\[
 \lambda\bigl((v+ht)(r+qt)\bigr)=vq+hr,
\]
the elements pairing integrally with $B$ form $Q^{-1}\oplus Ht$,
where $Q^{-1}=\Hom_H(Q,H)$ is viewed as a fractional ideal.
Associativity makes this identification $B$-linear. Finite-map duality
therefore gives
\[
 \omega_B\cong\Hom_H(B,H)\cong Q^{-1}\oplus Ht.
\]
The identity $(x):_H y=P$ gives $xQ^{-1}=P$. Multiplication by $x$
thus realizes the canonical module as
$W=P\oplus Ha=(x,z,a)B$.

Set $\mathfrak n=\mathfrak m_B=\mathfrak m_H\oplus Qt$.
The identities $PQ=x\mathfrak m_H$ and
$zQ\subseteq\mathfrak m_HP$ give
\[
 \mathfrak nW=\mathfrak m_HP\oplus\mathfrak m_Ha,
 \qquad W/\mathfrak nW\cong k^2\oplus k.
\]
Hence $\operatorname{type}(B)=3$.
Finally, multiplication by $y$ induces $C_0\cong Q/xH$, so
$C\cong C_0\oplus C_0b$, with $b^2=yw$. This gives the displayed
presentation and the free rank-two structure.
\end{proof}

\constructionstep{Gluing the wings}
For the other wing take the hypersurface $A=C[[u]]$, with quotient
$A/(u)=C=B/W$, and set $R=A\times_C B$. The canonical ideal $W$
makes this common ambient ring Gorenstein.

\begin{lemma}\label[lemma]{lem:canonical-ideal-gluing}
Let $B$ be a complete Cohen--Macaulay local ring of dimension $d\geq1$,
and let $K\subsetneq B$ be a canonical ideal.
Then $B/K$ is Gorenstein of dimension $d-1$, and
\[
 (B/K)[[u]]\times_{B/K}B\cong B[[u]]/uK[[u]]
\]
is Gorenstein of dimension $d$.
\end{lemma}

\begin{proof}
The canonical module is faithful, so $K$ lies in no associated prime
of $B$ and hence contains a nonzerodivisor by prime avoidance.
The depth lemma applied to $0\to K\to B\to B/K\to0$ shows that
$B/K$ is Cohen--Macaulay of dimension $d-1$. Canonical duality gives
$\Hom_B(K,K)=B$. Applying $\Hom_B(-,K)$ to the same sequence yields
\[
 \omega_{B/K}\cong\Ext_B^1(B/K,K)
 \cong\coker\bigl(K\longrightarrow\Hom_B(K,K)\bigr)
 \cong B/K.
\]
The map $B[[u]]\to(B/K)[[u]]\times_{B/K}B$, given by reduction
modulo $K$ and evaluation at $u=0$, is surjective with kernel
$uK[[u]]$. Since $K[[u]]$ is a canonical module of $B[[u]]$ and
multiplication by $u$ identifies it with $uK[[u]]$, the first assertion
applied to $B[[u]]$ proves the second.
\end{proof}

Applied to $B$ and $W$, the lemma gives the three-dimensional
complete Gorenstein ring $R\cong B[[u]]/uW[[u]]$. Its projection
kernels are $0\times W$ and $uA\times0$. Since $B$ and $A$ are
domains, these kernels annihilate one another and each is the full
annihilator of the other. They are prime, nonzero, and have zero
intersection. Thus $R$ is reduced, and $\Spec \, R$ is the union of
$\V(u)$ and $\V(WB[[u]])$ inside $\Spec \, B[[u]]$, with coordinate rings
$B$ and $A$, respectively.

\begin{proof}[Proof of \cref{thm:roberts-ma}]
It remains to construct the two modules over $R$.

\constructionstep{The modules on the double cover}
Let $U$ be the length-$15$ module constructed over the quadric $H$ in
\cref{lem:dhm-positive-interface}, and put
\[
 L=U\otimes_H B,\qquad L'=\Ext_B^3(L,B).
\]
Since $B$ is maximal Cohen--Macaulay over $H$ and $U$ has finite
projective dimension, $U\lotimes_H B\simeq L$. Thus $L$ has finite
length and finite projective dimension over $B$. Dualizing its finite
free resolution gives a resolution of $L'$, and dualizing back gives
$\Ext_B^3(L',B)\cong L$. Roberts' calculation
\cite[pp.~36--38]{Ro2} gives $\ell(L)=31$ and $\ell(L')=29$.

Put $E=C/(y,w)\cong k[b]/(b^2)$. Since $C$ is free over $C_0$,
\cref{eq:dhm-positive-interface} gives
\begin{equation}\label{eq:roberts-interface}
 L\lotimes_B C
 \simeq (U\lotimes_H C_0)\otimes_{C_0}C
 \simeq E^6\oplus\shift E^5.
\end{equation}
The sequence $y,w$ is regular on $C$, so
$\RHom_C(E,C)\simeq\shift^{-2}E$. Dualizing
\cref{eq:roberts-interface} therefore gives
\begin{equation}\label{eq:roberts-dual-interface}
 L'\lotimes_B C
 \simeq\shift^3\RHom_C(L\lotimes_B C,C)
 \simeq E^5\oplus\shift E^6.
\end{equation}

\constructionstep{The supporting wing of $M$ is $A$}
Take $M_B=L$ on the collapsed wing. On $A=C[[u]]$ take
\[
 N=A/(y,w)=E[[u]],\qquad M_A=N\oplus E^5,
\]
where $E=A/(u,y,w)$. Both summands have finite projective dimension.
Here $N\otimes_A C=E$ has the prescribed length $31-29=2$. The
element $u$ is regular on $N$ and annihilates $E$, so
\[
 M_A\lotimes_A C\simeq E^6\oplus\shift E^5\simeq M_B\lotimes_B C.
\]

\constructionstep{The collapsed wing of $M'$ is $A$}
On the supporting wing $B$ take
\[
 N'=B/(y,w+x)B,\qquad M'_B=N'\oplus L'.
\]
The sequence $y,w+x$ is part of a system of parameters on $H$, since
$H/(y,w+x)\cong k[[x,z]]/(x^2)$ has dimension one. It is therefore
regular on the maximal Cohen--Macaulay $H$-module $B$. Its image in
$C$ is the regular sequence $y,w$. It follows that $N'$ has dimension
and depth one, $\pdim_B N'=2$, and $N'\lotimes_B C\simeq E$.
In particular $\H^0_{\m}(M'_B)=L'$. On the collapsed wing $A$ take
$M'_A=E^6$, of finite projective dimension over $A$. By
\cref{eq:roberts-dual-interface} the restrictions agree over $C$, giving
\[
 M'_B\lotimes_B C\simeq E^6\oplus\shift E^6\simeq M'_A\lotimes_A C.
\]

\constructionstep{Patching the resolutions}
For each pair, patching the minimal free resolutions as in
\cref{subsec:variation} gives an $R$-module of finite projective
dimension with the prescribed restrictions. By \cref{prop:structure},
the spines are $N$ and $N'$, respectively.
\end{proof}

Both spines restrict to $E$ on $C$, with $\ell(E)=31-29=2$.
When $B$ is the collapsed wing, its contribution to the torsion
of $M$ is the submodule $(L')^\vee\subseteq\H^0_{\m}(M)$.
When $B$ is the supporting wing, its contribution appears as the
quotient $\H^0_{\m}(M')\twoheadrightarrow\H^0_{\m}(M'_B)=L'$.

\begin{example}
For the two modules just constructed, the torsion sequence in
\cref{prop:structure} gives
\[
\begin{aligned}
 \ell_R\bigl(\H^0_{\m}(M)\bigr)
   &=\ell_B(L')+5\ell_C(E)=29+5\cdot2=39,\\
 \ell_R\bigl(\H^0_{\m}(M')\bigr)
   &=6\ell_C(E)+\ell_B(L')=6\cdot2+29=41.
\end{aligned}
\]
Indeed, the end terms are $(L')^\vee$ and $E^{\oplus5}$ for $M$,
and $E^{\oplus6}$ and $L'$ for $M'$, by
\cref{lem:canonical-length}.
\end{example}


\section{Structure and size of the finite-length torsion}
\label{sec:torsion-size}

The torsion $T=\H^0_{\m}(M)$ is supported only at the closed point.
How can we recognize the contribution of each wing inside $T$?
Suppose $M$ has finite projective dimension over a Gorenstein ring
with linked Cohen--Macaulay wings. The contribution of the collapsed
wing $B=R/J$ consists precisely of the elements of $T$ killed by $J$.
The quotient is the finite-length torsion of the restriction to the
supporting wing. We use this description to show how the singularity
of $B$ constrains the length and Loewy length of $T$.

\begin{proposition}\label[proposition]{prop:intersection-dual}
Let $(R,\m,k)$ be a $d$-dimensional Gorenstein local ring and let
$B=R/J$ be Cohen--Macaulay of dimension $d$. Let $M$ be a nonzero
finitely generated $R$-module of finite projective dimension such that
$M_B=M\otimes_R B$ has finite length. For $T=\H^0_{\m}(M)$,
multiplication gives an exact sequence
\begin{equation}\label{eq:intrinsic-torsion}
  0\longrightarrow (0:_T J)\longrightarrow T
  \longrightarrow \Hom_R(J,T)\longrightarrow 0.
\end{equation}
If $J\ne0$, this is an essential nonsplit extension. Furthermore,
$(0:_T J)\cong D^\vee$, where $D=\Ext_B^d(M_B,B)$ is a nonzero
$B$-module of finite length and finite projective dimension.
\end{proposition}

\begin{proof}
Since $B$ is maximal Cohen--Macaulay over $R$, a minimal free
resolution of $M$ restricts to a minimal free resolution of $M_B$,
by the argument of \cref{obs:restriction}. Its length is $d$ by
Auslander--Buchsbaum over $B$. Dualizing the restricted resolution
computes $\Ext_B^i(M_B,B)\cong\Ext_R^i(M,B)$. These groups vanish
for $i<d$ because $M_B$ has finite length and $\depth B=d$. Thus the
reversed dual is a finite free resolution of $D$. Minimality gives
$D\ne0$, and $\Supp_B D\subseteq\Supp_B M_B$ gives finite length.

The long exact Ext sequence of $0\to J\to R\to B\to0$ now gives
\[
  0\longrightarrow \Ext_R^d(M,J)
  \longrightarrow \Ext_R^d(M,R)
  \longrightarrow D\longrightarrow0.
\]
Since $\pdim_R M=d$, taking the last cokernel in the dual resolution
identifies $\Ext_R^d(M,J)$ with $J\otimes_R\Ext_R^d(M,R)$, and the
first map is multiplication. By local duality,
$\Ext_R^d(M,R)^\vee\cong\H^0_{\m}(M)=T$. Thus Matlis duality and
tensor--Hom adjunction give a surjection $T\to\Hom_R(J,T)$ sending
$t$ to the map $j\mapsto jt$. Its kernel is $(0:_T J)\cong D^\vee$,
giving \eqref{eq:intrinsic-torsion}.

The kernel contains $\Soc(T)$ and hence is essential in $T$. If
$J\ne0$, composing $J\twoheadrightarrow k$ with a socle inclusion
$k\hookrightarrow T$ gives a nonzero element of $\Hom_R(J,T)$.
In this case the quotient is nonzero, so the extension is nonsplit.
\end{proof}

\begin{corollary}\label[corollary]{cor:torsion-socle}
Assume in addition the linked two-wing hypotheses of \cref{sec:setup},
with $A=R/I$ the supporting wing. Multiplication induces
\[
  M_A\cong\Hom_R(J,M),\qquad
  \H^0_{\m}(M_A)\cong\Hom_R(J,T).
\]
The intrinsic sequence \eqref{eq:intrinsic-torsion} is therefore
\begin{equation}\label{eq:canonical-torsion}
  0\longrightarrow I\otimes_B M_B\longrightarrow T
  \longrightarrow\H^0_{\m}(M_A)\longrightarrow0,
\end{equation}
with kernel $IM=(0:_T J)\cong D^\vee$.
\end{corollary}

\begin{proof}
Since $B$ is maximal Cohen--Macaulay over the Gorenstein ring $R$,
linkage gives $\RHom_R(B,R)\simeq I$. As $M$ is perfect,
tensor evaluation gives $\RHom_R(B,M)\simeq I\lotimes_R M$.
Consequently $\Ext_R^1(B,M)=0$, and evaluation at $1$ identifies
the image of $\Hom_R(B,M)\to M$ with $IM$. Applying $\Hom_R(-,M)$
to $0\to J\to R\to B\to0$ now gives
$M/IM\cong\Hom_R(J,M)$. Taking $\m$-torsion gives the second
isomorphism, since $J$ is finitely generated. The identification
$I\otimes_B M_B\cong IM$ follows from \cref{prop:structure}.
\end{proof}

We now use the inclusion $D^\vee\subseteq T$ to bound the length and
Loewy length of $T$. The index of a local ring $(B,\m_B,k)$ measures
its singularity cohomologically. In the singularity category
$D_{\mathrm{sg}}(B)=D^b(\operatorname{mod} B)/\operatorname{perf}(B)$,
it is given by
\[
  \operatorname{index}(B)
  =\min\bigl\{s\geq1 \mid B/\m_B^s\longrightarrow k
        \text{ is zero in }D_{\mathrm{sg}}(B)\bigr\}.
\]
Equivalently, the residue map lifts to a comparison of free
resolutions that vanishes in all sufficiently high degrees
\cite[Sections~1.3--1.4 and~2.1]{AV}. The index is finite and equals
one precisely when $B$ is regular \cite{Mar}.

\begin{theorem}\label[theorem]{thm:intersection-loewy}
Under the hypotheses of \cref{prop:intersection-dual},
\[
  \ll_R(T)\geq\operatorname{index}(B).
\]
If $\gr_{\m_B}(B)$ is Cohen--Macaulay, then also
\[
  \ll_R(T)\geq\operatorname{reg}\bigl(\gr_{\m_B}(B)\bigr)+1.
\]
If $J\ne0$, then $\ll_R(T)\geq2$.
\end{theorem}

\begin{proof}
Let $L$ be a nonzero finite-length $B$-module of finite projective
dimension and put $s=\ll_B(L)$. Choose $v\in L\setminus\m_B L$ and
a map $f:L\to k$ with $f(v)=1$. Since $\m_B^sL=0$, the residue map
factors as
\[
  B/\m_B^s \xrightarrow{\;1\mapsto v\;} L
  \xrightarrow{\;f\;} k.
\]
The composite is zero in $D_{\mathrm{sg}}(B)$ because $L$ is perfect.
Thus $\ll_B(L)\geq\operatorname{index}(B)$.
Compare with \cite[Lemma~1.3]{ABIM}.

Apply this to $L=D$. Matlis duality preserves annihilators, and
$D^\vee\subseteq T$, so $\ll_R(T)\geq\ll_B(D)$. This proves the
index bound. When $\gr_{\m_B}(B)$ is Cohen--Macaulay, applying
\cite[Theorem~B]{KP} to $D$ gives the regularity bound. Finally,
if $J\ne0$, the multiplication map in \eqref{eq:intrinsic-torsion}
has nonzero image. Hence $JT\ne0$, so $\m T\ne0$.
\end{proof}

Taking lengths in the same sequence gives a bound in terms of
multiplicity. Recall that the length conjecture predicts
$\ell_S(L)\geq e(S)$ for every nonzero finite-length module $L$ of
finite projective dimension over a local ring $S$
\cite[Conjecture~1]{IMW}. Here $e(-)$ denotes Hilbert--Samuel
multiplicity at the maximal ideal.

\begin{corollary}\label[corollary]{thm:collapsed-wing-torsion-size}
Under the hypotheses of \cref{prop:intersection-dual}, if the length
conjecture holds over $B$, then
\[
  \ell_R(T)\geq e(B)+\ell_R\bigl(\Hom_R(J,T)\bigr).
\]
In particular, $\ell_R(T)\geq e(B)+1$ when $J\ne0$. In the linked
setting of \cref{cor:torsion-socle}, if in addition the spine $N$
has finite projective dimension over $A$ and the length conjecture
holds over $A$, then
\[
  \ell_R(T)\geq e(A)+e(B)=e(R).
\]
\end{corollary}

\begin{proof}
Take lengths in \eqref{eq:intrinsic-torsion} and apply the length
conjecture to $D$. The quotient has positive length when $J\ne0$.
In the linked case, use \eqref{eq:canonical-torsion}. Its right-hand
term has finite projective dimension over $A$ by
\cref{cor:perfect-spine-balance}, so the length conjecture over $A$
bounds its length below by $e(A)$. The multiplicities add by
\eqref{eq:ring-patching}, since $\dim C=d-1$.
\end{proof}

\begin{remark}\label[remark]{rem:numerically-roberts}
The proof of Lech's conjecture \cite{OAI} gives a bound on Dutta
multiplicity. For a completed minimal $B$-free resolution $F$ of $D$,
one has $\chi_\infty(F)\geq\mu_B(D)e(B)$, including in mixed
characteristic \cite[Theorem~1.2 and Remark~2.10]{MaLech}. See also
\cite{IMW2} in equal characteristic. If $\widehat B$ is numerically
Roberts, then $\chi_\infty(F)=\chi(F)=\ell_B(D)$
\cite[Section~4]{IMW}. Consequently,
\[
  \ell_R(T)\geq\mu_B(D)e(B)+\ell_R\bigl(\Hom_R(J,T)\bigr).
\]

If $B$ is a complete intersection, its completion is numerically
Roberts and local duality over $B$ gives $D^\vee\cong M_B$. Thus
\eqref{eq:intrinsic-torsion} gives
\[
  \ell_R(T)=\ell_B(M_B)+\ell_R\bigl(\Hom_R(J,T)\bigr),
  \qquad \ell_B(M_B)\geq e(B).
\]
This proves Corollary~\ref{cor:intro-ci-length}, including the extra
one when $J\ne0$.
\end{remark}

A module of finite projective dimension with nonzero torsion can fail
the bounds above when the support condition is dropped. The module
in the following example has full support and its torsion is $k$.

\begin{example}\label[example]{ex:full-support-torsion}
Let $R$ be a singular Gorenstein local ring of dimension $d\geq1$.

\setcounter{constructionstep}{0}
\constructionstep{The module and its torsion}
For a minimal free resolution $F\to k$, write $(-)^*=\Hom_R(-,R)$
and set $M=\coker(F_{d-1}^*\to F_d^*)$. Since $\Ext_R^i(k,R)=0$
for $i<d$, the dual truncation gives a minimal free resolution of
$M$ of length $d$. Moreover, $\Ext_R^d(k,R)\cong k$ gives an exact
sequence
\[
  0\longrightarrow k\longrightarrow M
  \longrightarrow\operatorname{im}(F_d^*\to F_{d+1}^*)
  \longrightarrow0.
\]
The quotient is a submodule of a finite free module over a ring of
positive depth, so $\H^0_{\m}(M)\cong k$.

\constructionstep{Full support}
Away from $\m$, the resolution of $k$ is split exact. Thus $M$ is
locally free there, of constant rank given by the alternating sum of
the ranks in its finite free resolution. If that rank were zero,
$M$ would have finite length and hence $M=\H^0_{\m}(M)\cong k$.
This would give $k$ finite projective dimension over the singular
ring $R$. It follows that $\Supp_R M=\Spec \, R$.
\end{example}

\begin{question}\label[question]{q:torsion-size}
Which finite-length $R$-modules occur as $\H^0_{\m}(M)$ for a module
$M$ of finite projective dimension and unbalanced support? What
restrictions does this impose on the socle and the $\m$-adic filtration?

In the linked Gorenstein setting of \cref{cor:torsion-socle}, fix the
wings and the two module restrictions. Which extension classes in
\eqref{eq:canonical-torsion} arise as the gluing over $C$ varies?
What are the optimal lower bounds on length and Loewy length in terms
of the ambient ring and the collapsed wing?
\end{question}


\appendix
\section{The New Intersection Theorem via the socle trap}
\label{app:nit}

We give a big Cohen--Macaulay proof of the New Intersection Theorem
\cite{PS1,H,Ro1} by exhibiting a nonzero homology class.
For a related proof using complexes of maximal depth, see
Iyengar--Ma--Schwede--Walker \cite[Theorem~2.4]{IMSW}.

\begin{proposition}\label[proposition]{prop:big-cm-signal}
Let $(S,\n,k)$ be a Noetherian local ring, $\p\in\Spec \, S$, and
$e=\dim S/\p$. Let $F$ be a minimal complex of finite free $S$-modules
concentrated in degrees $0,\ldots,s$, with $F_s\neq0$ and
$\Supp F\cap\V(\p)=\{\n\}$. For any balanced big Cohen--Macaulay
$S/\p$-algebra $\mathcal B$, one has
\[
 \H_{s-e}(F\otimes_S\mathcal B)\neq0.
\]
In particular $s\geq e$.
\end{proposition}

\begin{proof}
The support hypothesis implies that the homology of
$F\otimes_S\mathcal B$ is $\n$-torsion. Since $F$ is a bounded complex
of finite free modules,
\[
 F\otimes_S\mathcal B
 \simeq F\otimes_S\mathrm R\Gamma_{\n}(\mathcal B).
\]
The balanced big Cohen--Macaulay property gives
\[
 \mathrm R\Gamma_{\n}(\mathcal B)
 \simeq\shift^{-e}\H^e_{\n}(\mathcal B).
\]
Regularity of a system of parameters on $\mathcal B$, together with
$\n\mathcal B\neq\mathcal B$, gives $\H^e_{\n}(\mathcal B)\neq0$.
Being a nonzero $\n$-torsion module,
$\H^e_{\n}(\mathcal B)$ has a nonzero socle.
Choosing $k\hookrightarrow\H^e_{\n}(\mathcal B)$ gives the
``socle trap''
\[
 0\neq F_s\otimes_S k\hookrightarrow
 \H_s(F\otimes_S\H^e_{\n}(\mathcal B))
 \cong\H_{s-e}(F\otimes_S\mathcal B).
\]
Minimality makes these socle elements cycles, and $F_{s+1}=0$ ensures
that none is a boundary. This proves the nonvanishing.
Since $F\otimes_S\mathcal B$ is concentrated
in nonnegative degrees, it also gives $s\geq e$.
\end{proof}

Balanced big Cohen--Macaulay algebras exist over every Noetherian local
ring by the work of Hochster and Huneke \cite{HH} in equal characteristic
and of Andr\'e \cite{Andre} in mixed characteristic.
Taking $F$ to resolve $M$ and $\p$ to define a
maximal-dimensional component of a closed set $Z$ meeting $\Supp M$ only
at $\n$ gives the New Intersection Theorem:
$\pdim_S M\geq\dim Z$.
For unbalanced support, take $\dim S/\p=\dim S$ and apply
Auslander--Buchsbaum to obtain \cref{eq:nit-reduction}.

\section*{Acknowledgements}

We thank Linquan Ma for suggesting the geometric reformulation of the
construction in \cref{thm:roberts-ma} and for pointing out the
relevance of \cref{ex:full-support-torsion}.
The identification of the top Ext dual in \cref{prop:intersection-dual}
was developed in a conversation with Srikanth Iyengar. We thank Andrew J. Soto Levins for valuable
comments on the exposition in the introduction.

\section*{AI disclosure}

Except for the specific AI contributions described below, the results
and proofs developed in this paper are due to the author. Before
\cite{Ma} appeared, I had been studying unbalanced support, had
established the positive-depth homological signals, and had developed
the torsion--spine perspective and the initial obstructions.
After the counterexample in \cite{Ma} appeared, ChatGPT identified
the length asymmetry in \cref{eq:length-asymmetry} when I supplied
the surrounding theory and the positivity obstruction. Building
on that observation, I developed the general numerical
torsion--spine balance using derived compatibility of the two wings.

I devised the variation in \cref{subsec:variation}. The length
asymmetry in \cref{eq:length-asymmetry} suggested similarly patching
Roberts' construction over a Gorenstein ambient ring. After I prompted
it to use \cite{Ro2}, ChatGPT found an algebraic construction. The geometric
reformulation of the construction in \cref{thm:roberts-ma} follows
a suggestion of Linquan Ma and was developed with ChatGPT's
assistance. In particular, \cref{prop:roberts-canonical,lem:canonical-ideal-gluing}
are due to ChatGPT. Gemini found the alternative Euler-characteristic proof
in \cref{rem:gemini}.

Prompted by my initial observations about the socle and the
torsion--spine structure, ChatGPT also identified and proved the socle
statement for the linked torsion sequence in \cref{cor:torsion-socle}.
I suggested describing the torsion on the supporting wing using
$\Hom_R(J,-)$. I developed the intrinsic formulation
\eqref{eq:intrinsic-torsion} with ChatGPT. I also developed the
shortened proof of \cref{thm:mixed-syzygy-nonvanishing} with ChatGPT.

ChatGPT produced an initial counterexample to the dimension inequality
for two modules of finite projective dimension. Theorem~\ref{thm:fivefold-pair}
presents a modified version of that example. The proof is mine, as is
the structural description of the modules and their correction terms.

ChatGPT constructed a five-dimensional counterexample to vanishing in
Section~\ref{subsec:intersection-multiplicities} using the
Dutta--Hochster--McLaughlin module \cite{DHM} over the three-dimensional
quadric. Using this, I then formulated the general construction of
Theorem~\ref{thm:intro-vanishing} and devised its proof strategy, including
the computation of the intersection
multiplicity.

I also used ChatGPT to help write, organize, edit, and format
the manuscript and prepare the bibliography. I take full responsibility
for the contents presented here.


\begin{thebibliography}{ABIM}
\setlength{\baselineskip}{9.5pt}
\bibitem[And]{Andre}
Y. Andr\'e, \emph{Weak functoriality of Cohen--Macaulay algebras},
J. Amer. Math. Soc. \textbf{33} (2020), no.~2, 363--380,
\href{https://doi.org/10.1090/jams/937}{doi:10.1090/jams/937}.

\bibitem[ABIM]{ABIM}
L.~L. Avramov, R.-O. Buchweitz, S.~B. Iyengar, and C. Miller,
\emph{Homology of perfect complexes}, Adv. Math. \textbf{223} (2010),
no.~5, 1731--1781,
\href{https://doi.org/10.1016/j.aim.2009.10.009}{doi:10.1016/j.aim.2009.10.009}.

\bibitem[AV]{AV}
L.~L. Avramov and O. Veliche,
\emph{Stable cohomology over local rings}, Adv. Math. \textbf{213}
(2007), no.~1, 93--139,
\href{https://doi.org/10.1016/j.aim.2006.11.012}{doi:10.1016/j.aim.2006.11.012}.

\bibitem[BE]{BE}
D.~A. Buchsbaum and D. Eisenbud,
\emph{What makes a complex exact?}, J. Algebra \textbf{25} (1973),
no.~2, 259--268,
\href{https://doi.org/10.1016/0021-8693(73)90044-6}{doi:10.1016/0021-8693(73)90044-6}.

\bibitem[DHM]{DHM}
S.~P. Dutta, M. Hochster, and J.~E. McLaughlin,
\emph{Modules of finite projective dimension with negative intersection
multiplicities}, Invent. Math. \textbf{79} (1985), no.~2, 253--291,
\href{https://doi.org/10.1007/BF01388973}{doi:10.1007/BF01388973}.

\bibitem[EV]{EV}
H. Esnault and E. Viehweg,
\emph{Lectures on vanishing theorems},
DMV Seminar, vol.~20,
Birkh\"auser Verlag, Basel, 1992,
\href{https://doi.org/10.1007/978-3-0348-8600-0}{doi:10.1007/978-3-0348-8600-0}.

\bibitem[GS]{GS}
H. Gillet and C. Soul\'e, \emph{K-th\'eorie et nullit\'e des
multiplicit\'es d'intersection}, C. R. Acad. Sci. Paris S\'er.~I Math.
\textbf{300} (1985), no.~3, 71--74.

\bibitem[H]{H}
M. Hochster, \emph{Topics in the homological theory of modules over
commutative rings}, CBMS Regional Conference Series in Mathematics,
no.~24, American Mathematical Society, Providence, RI, 1975,
\href{https://doi.org/10.1090/cbms/024}{doi:10.1090/cbms/024}.

\bibitem[HH]{HH}
M. Hochster and C. Huneke,
\emph{Applications of the existence of big Cohen--Macaulay algebras},
Adv. Math. \textbf{113} (1995), no.~1, 45--117,
\href{https://doi.org/10.1006/aima.1995.1035}{doi:10.1006/aima.1995.1035}.

\bibitem[24H]{24H}
S.~B. Iyengar, G.~J. Leuschke, A. Leykin, C. Miller, E. Miller,
A.~K. Singh, and U. Walther,
\emph{Twenty-Four Hours of Local Cohomology},
Graduate Studies in Mathematics, vol.~87,
American Mathematical Society, Providence, RI, 2007,
\href{https://doi.org/10.1090/gsm/087}{doi:10.1090/gsm/087}.

\bibitem[IMSW]{IMSW}
S.~B. Iyengar, L. Ma, K. Schwede, and M.~E. Walker,
\emph{Maximal Cohen--Macaulay complexes and their uses: A partial survey},
in \emph{Commutative algebra} (I. Peeva, ed.), Springer, Cham, 2021,
pp.~475--500,
\href{https://doi.org/10.1007/978-3-030-89694-2_15}{doi:10.1007/978-3-030-89694-2\_15}.

\bibitem[IMW]{IMW}
S.~B. Iyengar, L. Ma, and M.~E. Walker,
\emph{Multiplicities and Betti numbers in local algebra via lim Ulrich points},
Algebra Number Theory \textbf{16} (2022), no.~5, 1213--1257,
\href{https://doi.org/10.2140/ant.2022.16.1213}{doi:10.2140/ant.2022.16.1213}.

\bibitem[IMW2]{IMW2}
S.~B. Iyengar, L. Ma, and M.~E. Walker,
\emph{Lech's conjecture in equal characteristic}, preliminary draft (2026),
\href{https://mark-walker-22.github.io/LengthConjecture.Oct.6.pdf}{available online}.

\bibitem[KP]{KP}
N. KC and J. Pollitz,
\emph{Lower bounds on Loewy lengths of modules of finite projective dimension},
Adv. Math. \textbf{473} (2025), Paper No.~110309, 19~pp.,
\href{https://doi.org/10.1016/j.aim.2025.110309}{doi:10.1016/\allowbreak j.aim.2025.110309}.

\bibitem[KS]{KS}
N. KC and A.~J. Soto Levins,
\emph{On liftings of modules of finite projective dimension},
Int. Math. Res. Not. IMRN (2024), no.~24, 14729--14736,
\href{https://doi.org/10.1093/imrn/rnae257}{doi:10.1093/imrn/rnae257}.

\bibitem[Ma]{Ma}
L. Ma, \emph{Counterexamples to Peskine--Szpiro's conjectures on modules
of finite projective dimension}, preprint (2026),
\href{https://arxiv.org/abs/2608.24018}{arXiv:2608.24018}.

\bibitem[Ma2]{MaLech}
L. Ma, \emph{Lech's conjecture in all characteristics},
preliminary draft (2026),
\href{https://lquanma1019.github.io/Lech.pdf}{available online}.

\bibitem[Mar]{Mar}
A. Martsinkovsky, \emph{New homological invariants for modules over
local rings, I}, J. Pure Appl. Algebra \textbf{110} (1996), no.~1, 1--8,
\href{https://doi.org/10.1016/0022-4049(95)00092-5}{doi:10.1016/0022-4049(95)00092-5}.


\bibitem[MS]{MS}
C.~M. Miller and A.~K. Singh,
\emph{Intersection multiplicities over Gorenstein rings},
Math. Ann. \textbf{317} (2000), 155--171.



\bibitem[OAI]{OAI}
OpenAI, \emph{Lech's multiplicity conjecture}, preprint,
September~23, 2026,
\href{https://github.com/openai/math/blob/main/preprints/Lechs-multiplicity-conjecture-September-23-2026/paper.pdf}{available online}.

\bibitem[PS1]{PS1}
C. Peskine and L. Szpiro,
\emph{Dimension projective finie et cohomologie locale. Applications
\`a la d\'emonstration de conjectures de M. Auslander, H. Bass et
A. Grothendieck}, Publ. Math. Inst. Hautes \`Etudes Sci. \textbf{42}
(1973), 47--119,
\href{https://doi.org/10.1007/BF02685877}{doi:10.1007/BF02685877}.

\bibitem[PS2]{PS2}
C. Peskine and L. Szpiro,
\emph{Liaison des vari\'et\'es alg\'ebriques. I}, Invent. Math. \textbf{26}
(1974), 271--302,
\href{https://doi.org/10.1007/BF01425554}{doi:10.1007/BF01425554}.

\bibitem[Ro1]{Ro1}
P. Roberts, \emph{Le th\'eor\`eme d'intersection}, C. R. Acad. Sci. Paris
S\'er. I Math. \textbf{304} (1987), no.~7, 177--180.

\bibitem[Ro2]{Ro2}
P. Roberts, \emph{Intersection multiplicities in commutative algebra},
Publications de l'Institut de recherche math\'ematiques de Rennes,
S\'eminaires de math\'ematiques -- science, histoire et soci\'et\'e,
no.~4 (1985), 25--39,
\href{https://numdam.org/item/PSMIR_1985___4_25_0/}{Numdam}.

\bibitem[Ro3]{Ro3}
P. Roberts, \emph{The homological conjectures},
in \emph{Progress in Commutative Algebra 1: Combinatorics and Homology},
De Gruyter, Berlin, 2012, pp.~199--230,
\href{https://doi.org/10.1515/9783110250404.199}{doi:10.1515/9783110250404.199}.

\bibitem[Ro4]{Ro4}
P. Roberts, \emph{The vanishing of intersection multiplicities of
perfect complexes}, Bull. Amer. Math. Soc. (N.S.) \textbf{13} (1985),
no.~2, 127--130.

\bibitem[RS]{RS}
P. Roberts and V. Srinivas,
\emph{Modules of finite length and finite projective dimension},
Invent. Math. \textbf{151} (2003), 1--27.





\bibitem[Ser]{Ser}
J.-P. Serre, \emph{Local algebra},
Springer Monographs in Mathematics,
Springer-Verlag, Berlin, 2000,
\href{https://doi.org/10.1007/978-3-662-04203-8}{doi:10.1007/978-3-662-04203-8}.

\bibitem[Sta]{Stacks}
The Stacks Project Authors, \emph{The Stacks Project},
\href{https://stacks.math.columbia.edu/tag/00MF}{Tag~00MF} and
\href{https://stacks.math.columbia.edu/tag/0H7Z}{Tag~0H7Z}.

\end{thebibliography}
\end{document}